\documentclass[11pt]{amsart}
\usepackage{amssymb,upref}
\usepackage{enumerate, graphicx, float}

\theoremstyle{plain}

\newtheorem{theor}{Theorem}
\newtheorem{lem}[theor]{Lemma}
\newtheorem{prop}[theor]{Proposition}
\newtheorem{cor}[theor]{Corollary}

\theoremstyle{definition}
\newtheorem{dfn}[theor]{Definition}
\newtheorem{example}[theor]{Example}
\newtheorem{exe}[]{Exercise}
\newtheorem{rem}[theor]{Remark}
\def\qed {{
   \parfillskip=0pt        
   \widowpenalty=10000     
   \displaywidowpenalty=10000  
   \finalhyphendemerits=0  
   \leavevmode             
   \unskip                 
   \nobreak                
   \hfil                   
   \penalty50              
   \hskip.2em              
   \null                   
   \hfill                  
   $\square$
  
   \par}}                  
\newcommand{\thmref}[1]{Theorem~\ref{#1}}
 
 \newcommand{\corref}[1]{Corollary~\ref{#1}}
  \newcommand{\propref}[1]{Proposition~\ref{#1}}

\newcommand{\figref}[1]{Figure~\ref{#1}}

\def\XXint#1#2#3{{\setbox0=\hbox{$#1{#2#3}{\int}$ }
\vcenter{\hbox{$#2#3$ }}\kern-.6\wd0}}

\def\tr{\mathop{\rm tr}}
\def\div{\mathop{\rm div}}

\newcommand{\ds}{\displaystyle}

\newcommand{\dl}{{\delta}}

\newcommand{\K}{\mathbb{K}}
\newcommand{\bee}{\begin{equation*}}
\newcommand{\eee}{\end{equation*}}
\newcommand{\be}{\begin{equation}}
\newcommand{\ee}{\end{equation}}
\newcommand{\baa}{\begin{align*}}
\newcommand{\eaa}{\end{align*}}
\newcommand{\ba}{\begin{align}}
\newcommand{\ea}{\end{align}}

\DeclareMathOperator*{\esssup}{ess\,sup\,}

\DeclareMathOperator*{\essinf}{ess\,inf\,}
\newcommand{\eps}{\varepsilon}

\newcommand{\B}{\mathcal{B}_\Omega}
\newcommand{\na}{\mathbb{N}}
\newcommand{\re}{\mathbb{R}}
\newcommand{\rn}{\mathbb{R}^n}

\newcommand{\R}{{\mathbb R}}

\newcommand{\Rn}{{\mathbb{R}}^{n}}

\begin{document}


\thanks{Second and third authors were partially supported by the NSF under grant DMS 0901587.}

\address{Sapto Indratno, Bandung Institute of Technology, Statistics Research Group, Ganesha 10, Bandung-40132, Indonesia.} \email{sapto@math.itb.ac.id}

\address{Diego Maldonado, Kansas State University, Department of Mathematics. 138 Cardwell Hall,
Manhattan, KS-66506, USA.} \email{dmaldona@math.ksu.edu}

\address{Sharad Silwal, Northland College, 1411 Ellis Avenue, Ashland, WI 54806, USA.} \email{ssilwal@northland.edu}

\title[A visual formalism for reverse classes]{A visual formalism for weights satisfying reverse inequalities}
\author{Sapto Indratno, Diego Maldonado, and Sharad Silwal}

\date{\today}

\maketitle

\date{\today}

\begin{abstract}
In this expository article we introduce a diagrammatic scheme to represent reverse classes of weights and some of their properties. 
\end{abstract}

\section{Introduction}\label{secintro}

Let $B\subset\Rn$ be a Euclidean ball and let $|B|$ denote its
Lebesgue measure. Given a function $w>0$ defined on $B$, H\"older's (or Jensen's)
inequality implies that for any  $-\infty\leq p \leq q\leq \infty$ we have
\be\label{natural} \left(\frac{1}{|B|}\int_B
w^p(x)\,dx \right)^{1/p}\leq \left(\frac{1}{|B|}\int_B w^q(x)
\,dx\right)^{1/q},
\ee 
finite or infinite (see Definition \ref{rmean} below for the meaning of these integrals when $p$ or $q$ equal $0, -\infty$, or $\infty$). It is a remarkable fact that for large families of weights $w$ one or more of the natural
inequalities in \eqref{natural} can be reversed, up to a
multiplicative constant, uniformly over vast collections of balls. 

The purpose
of this expository article is to formalize the study of such weights through a
diagrammatic scheme.  We found this visual scheme to be of pedagogical value when teaching or explaining Muckenhoupt weights as well as some topics on elliptic PDEs such as Harnack's inequality and
Moser's iterations. It also serves as a mnemonic device for
remembering various results relating weights in reverse classes and as a tool for posing
questions and conjectures for such weights.

As part of the exposition we will review a number of results on the theory of Muckenhoupt weights; however, we will make no attempt at accounting for the history of weights satisfying reverse inequalities; readers are encouraged to consult, for instance, \cite[Chapter 7]{Duo}, \cite[Chapter IV]{GCRdF}, \cite[Chapter 9]{Grafakos}, \cite[Chapter 2]{Jou83}, \cite[Chapter 5]{St}, \cite[Chapters I and II]{StrTor}, \cite[Chapter IX]{Tor}.

Since inequalities such as \eqref{natural} require only the notions of `ball' and `integral', we set up the rest of the article in the
context of spaces of homogeneous type, also known as doubling quasi-metric spaces.

\begin{dfn} A \emph{quasi-metric} space is a pair $(X,d)$ where X is a
non-empty set and $d$ is a \emph{quasi distance} on $X$, that is, $d : X
\times X \to [0,\infty)$ such that
\begin{enumerate}[(i)]
\item $d(x,y)=d(y,x)$ for all $x, y \in X$,
\item $d(x,y)=0$ if and only if $x=y$, and
\item there exists $K\geq 1$ (quasi-triangle constant), such that
\begin{equation}\label{quasitrianK}
d(x,y)\leq K (d(x,z)+d(z,y)) \quad \forall x, y, z \in X.
\end{equation}
\end{enumerate}
Let $(X,d)$ be a quasi-metric space. The $d$-ball with center $x\in
X$ and radius $r>0$ in $(X,d)$ is defined by 
$$
B_r(x):=\{y\in X: d(x,y)<r\}. 
$$
For $B:=B_r(x)$ and $\lambda > 0$, $\lambda B$
denotes the ball $B_{\lambda r}(x)$. 

Let $\mu$ be a measure defined
on the balls of $X$. The triple $(X,d,\mu)$ is said to be a \emph{space of
homogeneous type} (as introduced by Coifman and Weiss \cite[Chapter III]{CW71}) 
if $(X,d)$ is a quasi-metric space and
$\mu$ satisfies the \emph{doubling property}, that is, if there
exists a positive constant $C_\mu>1$ such that 
\begin{equation}\label{mudoub}
0<\mu(B_{2r}(x))\leq C_\mu \, \mu(B_r(x)) \quad \forall x\in X, r>0. 
\end{equation}
Any constant depending only on the quasi-triangle constant $K$ in \eqref{quasitrianK} and the doubling constant $C_\mu$ in \eqref{mudoub} will be called a \emph{geometric
constant}. For basic topological and measure theoretic results on spaces of homogeneous type, such as the existence of a constant $\alpha \in (0,1)$, depending only on $K$, and a distance $\rho$ in $X$ with 
$$
\frac{\rho(x,y)}{2} \leq d(x,y)^\alpha \leq 4 \rho(x,y) \quad \forall x, y \in X,
$$
or the fact that the atoms for $(X,d, \mu)$ must be countable and isolated (where $x \in X$ is an atom if $\mu(\{x\}) > 0$), see the pioneering work of Mac\'ias and Segovia \cite{MS79}. Another classical reference, which includes a long list of examples of spaces of homogeneous type, is the survey article by Coifman and Weiss \cite{CW77}. For an exposition of some of the topics in this survey in the Euclidean context but without involving the doubling condition, see \cite{OP02}.\\
\end{dfn}

\begin{dfn}\label{rmean} Given a function $w>0$, a ball $B$, and
$- \infty \leq p \leq \infty$, set
\begin{equation}\label{pmeanB}
w(p,B):=\left(\frac{1}{\mu(B)}\int_B w^p\,d\mu\right)^{1/p},
\end{equation}
whenever finite, to be the \emph{$p$-mean} of $w$ over $B$. Some special cases of \eqref{pmeanB} are noteworthy. The
$1$-mean of $w$ over $B$ is the arithmetic mean, or simply the
average, of $w$ over $B$. When $p=0$, $w(0,B)$ takes on the form
\[
w(0,B):= \lim\limits_{p \rightarrow 0} w(p,B) = \exp\left(\frac{1}{\mu(B)}\int_B \ln w\,d\mu\right),
\]
which is the geometric mean of $w$ over $B$. With the exponent $p=-1$ we get $w(-1,B)$, the harmonic mean of $w$ over $B$, and the exponents $p=-\infty$ and $p=\infty$ yield the essential infimum and essential supremum of $w$ over $B$, respectively, that is,
$$
w(-\infty, B) = \essinf\limits_{B} w \quad \text{ and } \quad w(\infty, B) = \esssup\limits_{B} w.
$$
\end{dfn}

\section{Reverse Classes and their diagrams}\label{secrcdef}

Let $(X,d,\mu)$ be a space of homogeneous type and let $\Omega \subset X$ be an open set. Throughout the article we will consider balls in the collection
\begin{equation}\label{defBOmega}
\B:=\{B\subset\Omega\mid 2B\subset\Omega\}.
\end{equation}

\begin{dfn}\label{RSC} Let $-\infty\leq r<s\leq \infty$ and $C \geq 1$. We write $w\in RC(r,s,C)$ and say that $w$ is in the reverse class
with exponents $r$ and $s$ if $w^r \in L^1_{loc}(\Omega)$ and 
\begin{equation}\label{rcdef}
w(r,B) \leq w(s,B) \leq C w(r,B) \quad \forall B \in \B.
\end{equation}
The smallest constant $C \geq 1$ validating \eqref{rcdef} will be called the \emph{reversal constant} of the weight $w$ in the class $RC(r,s)$ and it will be denoted by $[w]_{RC(r,s)}$. Also, define
$$
RC(r,s) := \bigcup\limits_{C \geq1} RC(r,s,C).
$$
The fact that a weight $w$ belongs to a reverse class $RC(r,s)$ will be represented by the following diagram
\begin{figure}[H] 
   \centering
   \includegraphics[width=5in]{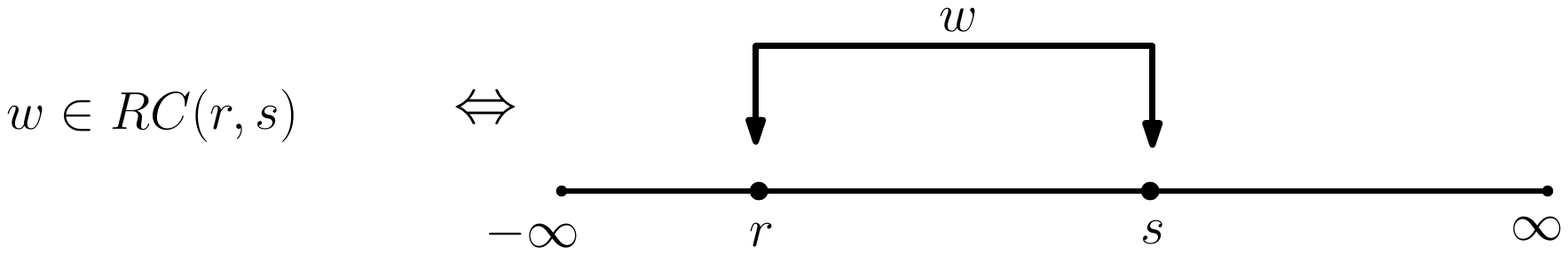}
   \caption{Illustration of the reverse class $RC(r,s)$, indicating the uniform comparability between $r$- and $s$-means.}
   \label{rc}
\end{figure}
In order to allow for more clear and less crowded pictures, we will make no distinctions between arrows drawn above or below the extended real line. That is,
\begin{figure}[H] 
   \centering
   \includegraphics[width=4.7in]{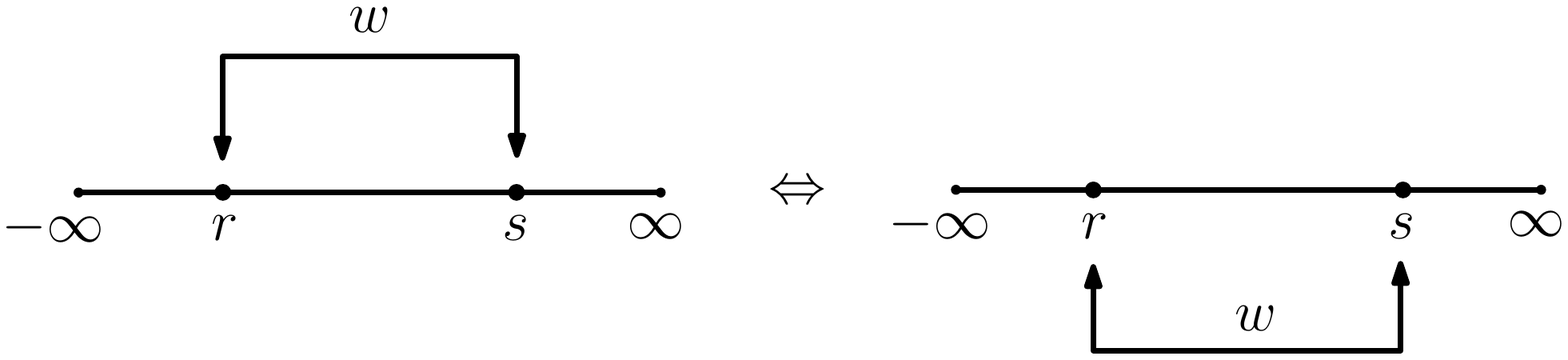}
   \caption{No distinction is made between arrows drawn above or below the extended real line.}
   \label{rc}
\end{figure}

In all rigor, we should write  $w\in RC(r,s,C, \Omega)$ to account for the open set $\Omega$ where the balls are taken, but $\Omega$ will be always understood from the context.
\end{dfn}

\begin{dfn}\label{wRSC} Let $-\infty\leq r<s\leq \infty$ and $C>0$. We write $w\in RC^{weak}(r,s,C)$, or
simply, $w\in RC^{weak}(r,s)$, and say that $w$ is in the weak reverse
class with exponents $r$ and $s$ if $w^r \in L^1_{loc}(\Omega)$ and 
$$
\left(\frac{1}{\mu(B)}\int_B w^s \,d\mu\right)^{1/s}\leq
C\left(\frac{1}{\mu(B)}\int_{2B} w^r \,d\mu\right)^{1/r} \quad \forall B \in \B. 
$$
We visually represent this condition by joining the two exponents $r$ and $s$ by a dashed straight line with arrowheads
pointing at them as illustrated in Figure \ref{rcweak}.
\begin{figure}[H] 
   \centering
   \includegraphics[width=5in]{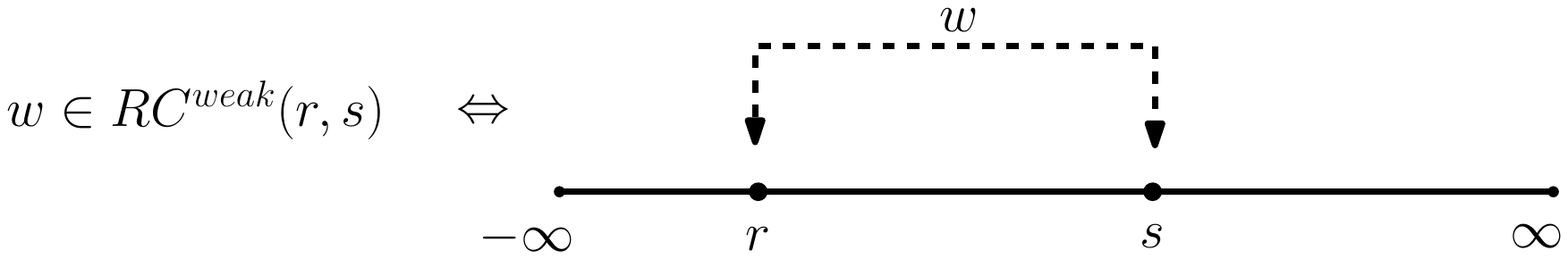}
   \caption{The class $RC^{weak}(r,s)$. }
   \label{rcweak}
\end{figure}
\end{dfn}

\begin{rem}\label{rem:dashedtosolid}
Notice that if  $w\in RC^{weak}(r,s)$ and if $w^r$ is a doubling weight, in the sense that there exists a constant $C > 0$ such that $w^r(1,2B) \leq C w^r(1,B) $ for every $B \in \B$, then $w\in RC(r,s)$ and the dashed line in Figure \ref{rcweak} becomes a solid one.
\end{rem}

\section{The three axioms of the visual formalism}\label{secprop}

In this section we introduce the three main properties of the visual formalism: the shrinking property, the concatenation property, and the scaling property. These  properties will be deduced from the definition of the reverse classes; however, they will be used as axioms during formal manipulations.

\subsection{The shrinking property} This property says that the arrows representing reverse classes can always be shrunk without worsening reversal constants. More precisely,

\begin{lem}\label{lemmashrink} Fix any $-\infty \leq r\leq \tilde{r}<\tilde{s}\leq s \leq \infty$.  If $w\in RC(r,s,C)$, then $w\in
RC(\tilde{r},\tilde{s},C)$. Visually,
\begin{figure}[H] 
   \centering
   \includegraphics[width=5.2in]{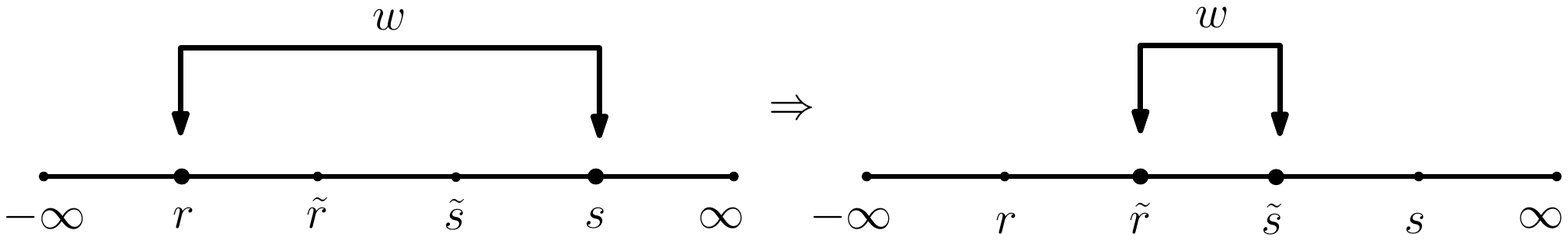}
   \caption{The shrinking property, that is, $RC(r,s,C) \subset RC(\tilde{r},\tilde{s},C)$. In particular, $[w]_{RC(\tilde{r},\tilde{s})} \leq [w]_{RC(r,s)}$.}
   \label{freerc}
\end{figure}
\end{lem}

\begin{proof} Since  $\tilde{s}\leq s$ and $r\leq \tilde{r}$ we have natural inequalities
\begin{equation}\label{stildetos}
w(\tilde{s}, B)\leq w(s, B) \quad \text{ and } \quad  w(r,B) \leq w(\tilde{r},B).
\end{equation} 
Now, combining \eqref{stildetos} with the defining
inequality of $w\in RC(r,s,C)$, we obtain $ w(\tilde{s}, B)\leq w(s,B) \leq C w(r,B) \leq C w(\tilde{r},B).$
\end{proof}

\subsubsection{The splitting property} As a consequence of the shrinking property, it follows that: for every $-\infty \leq r < s < t \leq \infty$ we obtain that $RC(r,t, C) \subset RC(r,s,C) \cap RC(s,t,C)$; in particular, $\max\{ [w]_{RC(r,s)},[w]_{RC(s,t)} \} \leq [w]_{RC(r,t)}$. We call this the splitting property and visually represent it with Figure \ref{fig:split}.
\begin{figure}[H] 
   \centering
   \includegraphics[width=4.5in]{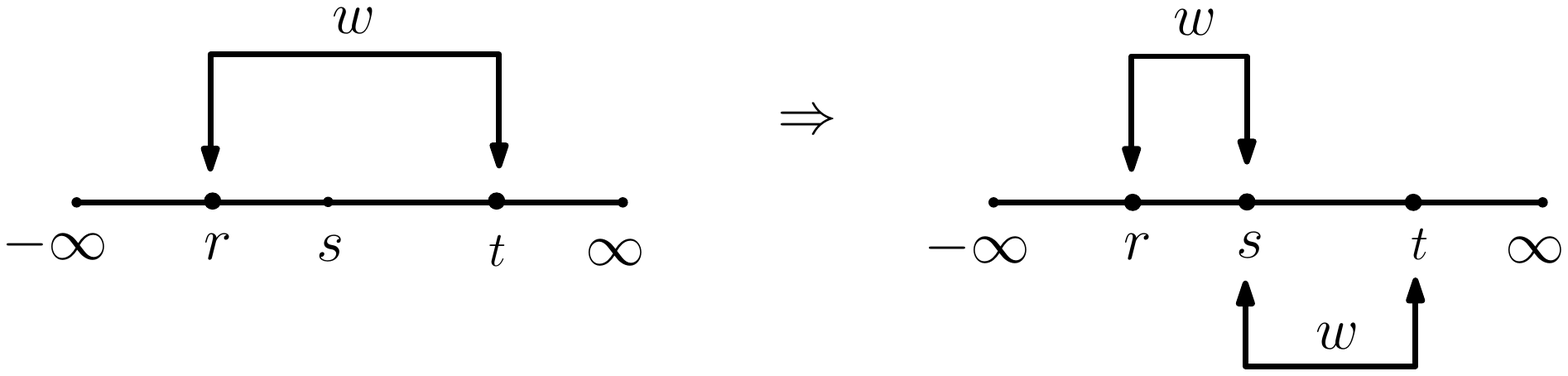}
   \caption{The splitting property:  $RC(r,t, C) \subset RC(r,s,C) \cap RC(s,t,C)$ and $\max\{ [w]_{RC(r,s)},[w]_{RC(s,t)} \} \leq [w]_{RC(r,t)}$.}
   \label{fig:split}
\end{figure}

As shown, the arrows representing reverse classes can always be shrunk as much as desired. However, they cannot be stretched as much as we wish. As we will see  in Section \ref{secc:selfimprove}, stretching the arrows will describe self-improvement properties for reverse classes.

\subsection{The concatenation property} This property provides the converse to the splitting property. It says that whenever an arrow begins where another ends, then they can be concatenated. Moreover, the reversal constant for the concatenation is controlled by the product of the reversal constants involved. 

\begin{lem}\label{freelemmaa} Fix any $-\infty \leq r < s < t \leq \infty$. If $w\in RC(r,s,C_1)\cap RC(s,t,C_2)$, then $w\in
RC(r,t,C_1C_2)$. Consequently,  $RC(r,s)\cap RC(s,t) = RC(r,t)$.
\end{lem}
 Bearing in mind that there is no distinction between arrows drawn above or below the extended real line, the statement of Lemma \ref{freelemmaa} can be visually recast as follows:
\begin{figure}[H] 
  \centering
   \includegraphics[width=5.4in]{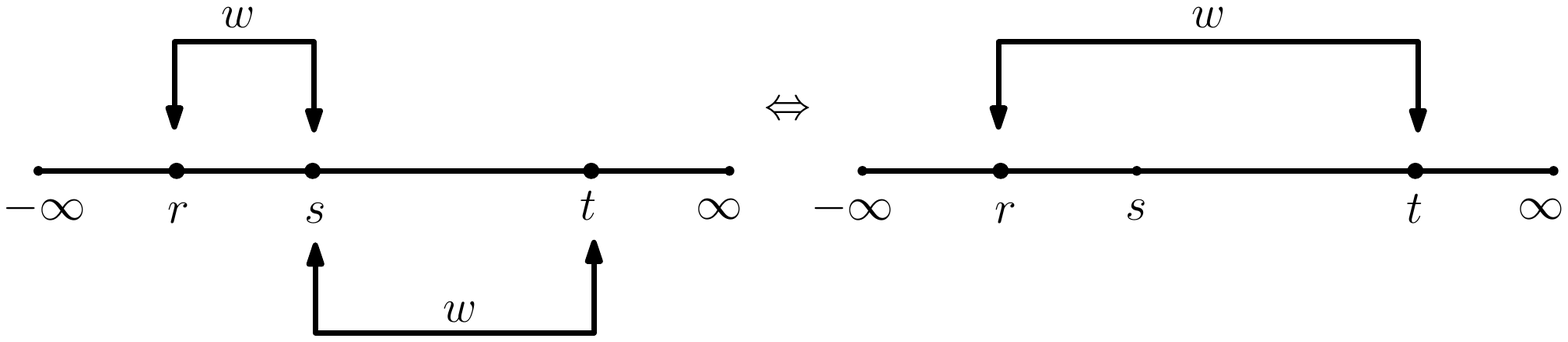}
   \caption{The concatenation property. Quantitatively, $[w]_{RC(r,t)} \leq [w]_{RC(r,s)}[w]_{RC(s,t)}$.}
   \label{rcconnect}
\end{figure}
Notice that the concatenation property together with the splitting property implies that $RC(r,s)\cap RC(s,t) = RC(r,t)$.

\emph{Proof of Lemma \ref{freelemmaa}}. Since $w\in
RC(r,s,C_1)\cap RC(s,t,C_2)$, we have
$
w(t,B) \leq C_2 w(s, B) \leq C_1 C_2 w(r, B). $
 \qed

\subsection{The scaling property}
This property dictates how reverse classes and reversal constants  behave under the operation of taking powers. More precisely,
\begin{theor}\label{rcmain}
Let $- \infty \leq r < s \leq \infty$ and fix $w\in RC(r,s,C)$. Then,
\begin{enumerate}[(i)]
\item\label{rcmain1} for any $\theta>0$,  we have $w^\theta\in
RC(\frac{r}{\theta},\frac{s}{\theta},C^\theta)$,
\item\label{rcmain2} for any $\theta
<0$, we have $w^\theta \in RC(\frac{s}{\theta},\frac{r}{\theta},C^{|\theta|})$.
\end{enumerate}
Visually, this is illustrated in Figure \ref{fig:rcraisedtwo}.
\begin{figure}[H] 
   \centering
   \includegraphics[width=5.2in]{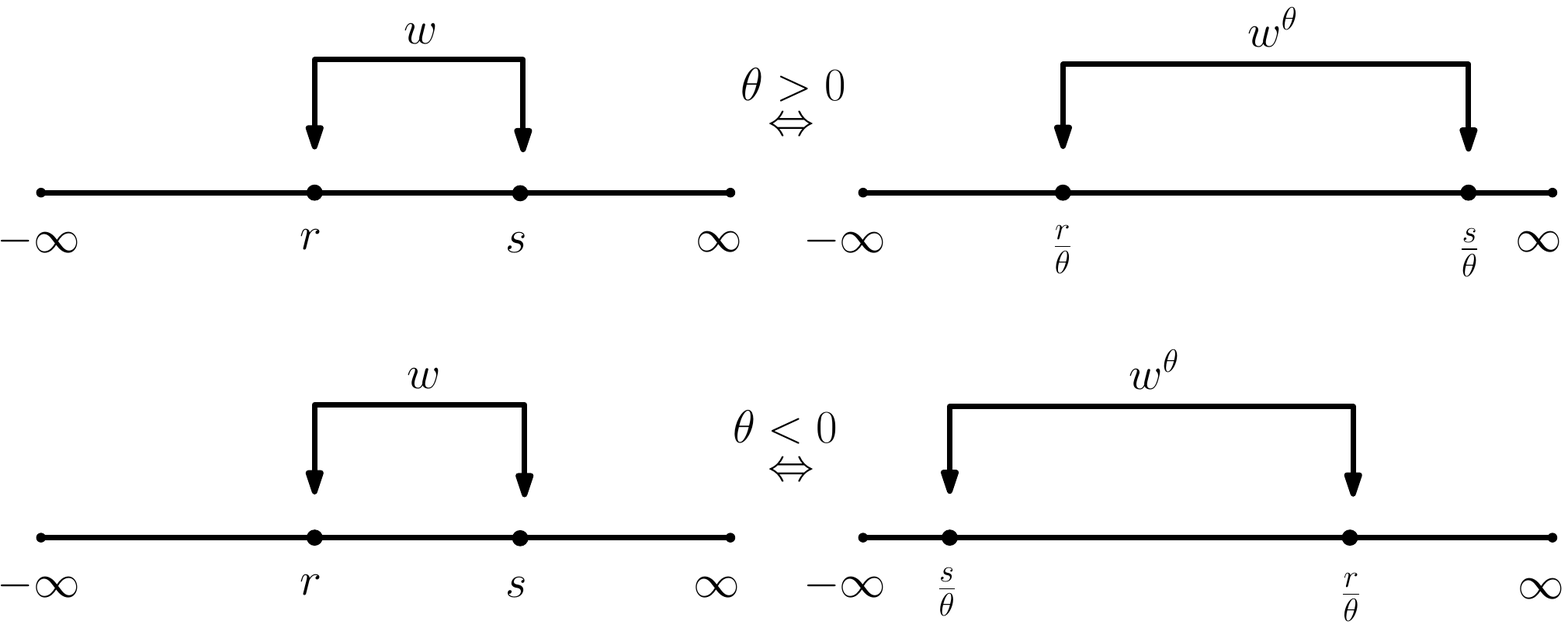}
   \caption{The scaling property. Quantitatively, $[w^\theta]_{RC(r/\theta, s/\theta)} = [w]_{RC(r,s)}^\theta$, for $\theta > 0$; and $[w^\theta]_{RC(s/\theta, r/\theta)} = [w]_{RC(r,s)}^{|\theta|}$, for $\theta < 0$.}
   \label{fig:rcraisedtwo}
\end{figure}
\end{theor}

\begin{proof}
Let us prove \eqref{rcmain1}. Since $w\in RC(r,s,C)$, given $B \in \B$ we have
\be\label{mi} 
w(s,B)=\left(\frac{1}{\mu(B)}\int_B
w^{\theta\frac{s}{\theta}} \,d\mu\right)^{1/s}\leq C w(r,B) =
C\left(\frac{1}{\mu(B)}\int_B w^{\theta \frac{r}{\theta}}\,d\mu
\right)^{1/r}. 
\ee 
For $\theta>0$ we get
$$
\left(\frac{1}{\mu(B)}\int_B w^{\theta\frac{s}{\theta}}\,d\mu
\right)^{\theta/s}\leq C^\theta\left(\frac{1}{\mu(B)}\int_B
w^{\theta \frac{r}{\theta}} \,d\mu\right)^{\theta/r}, 
$$
that is, $w^\theta\in RC(\frac{r}{\theta},\frac{s}{\theta},C^\theta).$ To prove \eqref{rcmain2}, we use \eqref{mi} again, but now with $\theta < 0$ 
$$
\left(\frac{1}{\mu(B)}\int_B 
w^{\theta\frac{s}{\theta}} \,d\mu\right)^{\theta/s}\geq
C^\theta\left(\frac{1}{\mu(B)}\int_B w^{\theta
\frac{r}{\theta}}\,d\mu \right)^{\theta/r}, 
$$ 
that is, $w^\theta\in RC(\frac{s}{\theta},\frac{r}{\theta},C^{|\theta|})$.
\end{proof}

\begin{rem}\label{rem:qq} It is convenient to point out that, during the formal manipulations, the axioms above provide qualitative as well as quantitative information about the reverse classes and their weights. Meaning that when a particular weight is subjected to a sequence of axioms of the visual formalism, there will always be an explicit control on the reversal constants. 
\end{rem}

\section{Some well-known reverse classes}

\begin{dfn}
Let $1<s\leq\infty$. A weight $w$ is said to belong to $RH_s$, the reverse H\"older class of order $s$, 
if the following inequality holds 
\bee
[w]_{RH_s}:=\ds\sup_{B\in\B} w(s,B)w(1,B)^{-1}  < \infty. 
\eee
In other words, $RH_s = RC(1,s)$. Visually,
\begin{figure}[H]    
\centering
   \includegraphics[width=4.7in]{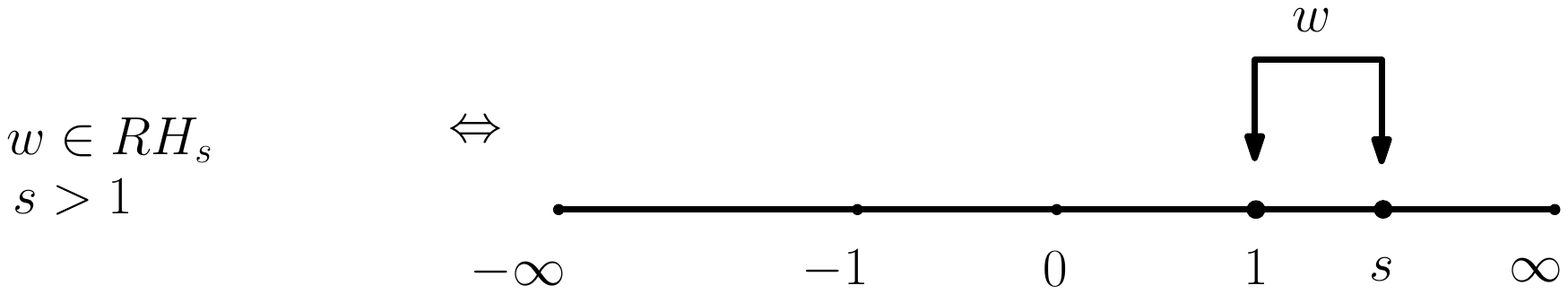}
   \caption{The $RH_s$ classes. Notice that $RH_s = RC(1,s)$.}
      \label{rhs}
\end{figure}
\end{dfn}

\begin{dfn}\label{ap} Let $1< p<\infty$. A weight $w$ is said to belong to  the Muckenhoupt class $A_p$ if 
\bee\label{ap}
[w]_{A_p}:=\ds\sup_{B\in\B} w(1,B) w\left(\frac{1}{1-p}, B\right)^{-1}< \infty. 
\eee
In other words,  $ A_p= RC(1/(1-p),1)$. We write $w\in A_1$  if 
\bee\label{ap}
[w]_{A_1}:=\ds\sup_{B\in\B} w(1,B) w(-\infty, B)^{-1}< \infty.
\eee
That is, $A_1 = RC(-\infty, 1)$. Finally, we write $w\in A_\infty$  if 
\begin{equation}\label{ainfH}
[w]_{A_\infty}:=\ds\sup_{B\in\B} w(1,B) w(0,B)^{-1}< \infty.
\end{equation}
That is, $A_\infty = RC(0,1)$. Again, we have chosen to use the notation $A_p$ instead of $A_p(\Omega)$. Muckenhoupt classes can be visually represented by the following diagrams in Figure \ref{ap}.
\begin{figure}[H] 
   \centering
   \includegraphics[width=4.5in]{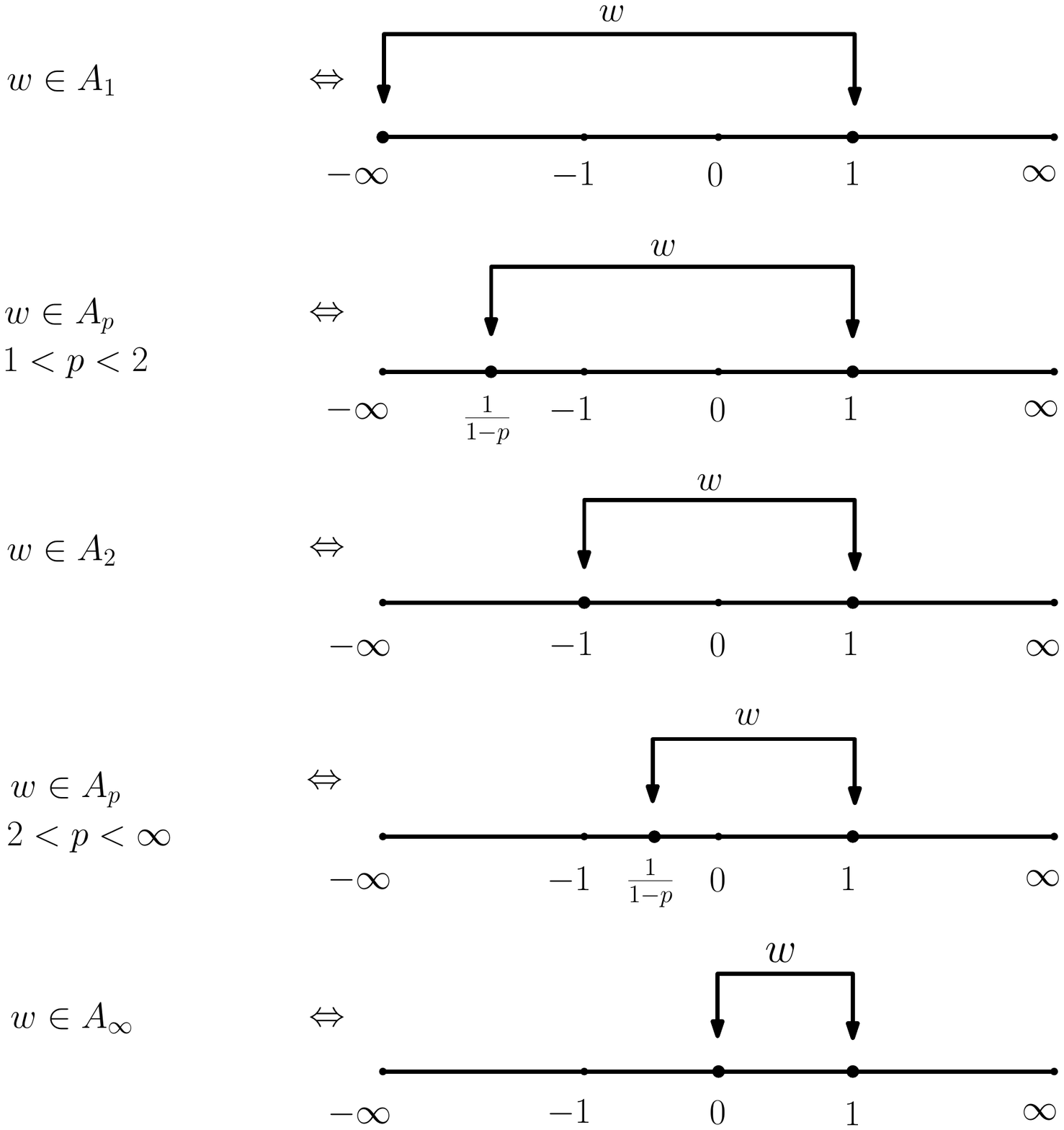}
   \caption{The $A_p$ classes as reverse classes: $A_p = RC((1-p)^{-1}, 1)$.}
   \label{ap}
\end{figure}
\end{dfn}

Again, we refer the reader to \cite[Chapter 7]{Duo}, \cite[Chapter IV]{GCRdF}, \cite[Chapter 9]{Grafakos}, \cite[Chapter 2]{Jou83}, \cite{Muck72}, \cite[Chapter 5]{St}, \cite[Chapters I and II]{StrTor}, \cite[Chapter IX]{Tor}, as well as references therein, for all the basic material and history of reverse H\"older classes and Muckenhoupt weights.

\begin{rem} The constant $[w]_{A_\infty}$ in \eqref{ainfH} was introduced by Hru\v{s}\v{c}ev in \cite{Hru84}. Another constant, which has lately received a lot of attention, given by
\begin{equation}
[w]'_{A_\infty}:= \sup\limits_{B \in \B} \frac{1}{\mu(B)} \int_B \mathcal{M}(w \chi_B) \, d\mu,
\end{equation}
where $\mathcal{M}$ stands for the Hardy-Littlewood maximal operator, was introduced by Fujii \cite{Fujii} and Wilson \cite{W}. The constant $[w]'_{A_\infty}$ also characterizes $A_\infty$. More precisely, in the Euclidean setting (with $\Omega = \rn$), Hyt\"onen and P\'erez proved in \cite{HP} that one has $[w]'_{A_\infty} \leq c_n [w]_{A_\infty}$, for some dimensional constant $c_n > 0$. A corresponding inequality was proved in the context of spaces of homogeneous type by Hyt\"onen, P\'erez, and Rela in \cite{HPR}. The study of $A_\infty$ through the constant $[w]'_{A_\infty}$ seems to be better suited for obtaining sharp reversal constants and reversal exponents, see \cite{HP, HPR}, and references therein. 
\end{rem}

\subsection{Some practice} With the purpose of getting some practice with the visual formalism, next we go over some well-known properties of $A_p$ weights (see, for instance, \cite[Section 1]{JN91} and \cite{CUN95}). 

It will be useful to bear in mind Remark \ref{rem:qq} as well as the facts that, for $1 < p < \infty$,  $A_p = RC((1-p)^{-1}, 1)$, $A_1 = RC(-\infty, 1)$, and $A_\infty = RC(0,1)$. 

\begin{cor}\label{pinq} Fix $1<p <\infty$ and let $p'$ denote its H\"older conjugate, that is, $\frac{1}{p}+\frac{1}{p'}=1$. Then $w\in A_p$ if and only if $w^{1-p'}\in A_{p'}$. Moreover, $[w^{1-p'}]_{A_{p'}}=[w]^{p'-1}_{A_p}$.
\end{cor}

\begin{proof} This proof uses only the scaling property along with the fact that $(1-p)(1-p')=1$.
\begin{figure}[H] 
   \centering
   \includegraphics[width=5in]{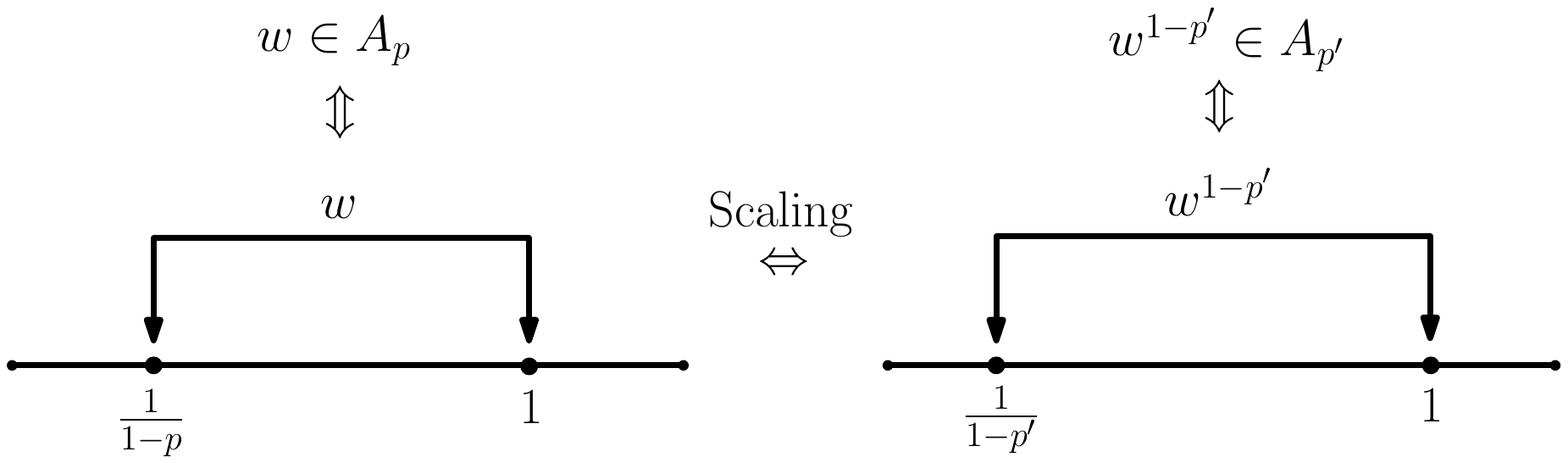}
   \caption{Visual proof of Corollary \ref{pinq}}
   \label{fig:pinq}
\end{figure}
Notice that, since scaling by a power $\theta$ has the effect of raising reversal constants to the power $|\theta|$ (see \thmref{rcmain}), for the sequence of steps, from left to right, in Figure \ref{fig:pinq} we have scaled by $\theta = 1-p' <0$, yielding
$$
 [w^{1-p'}]_{A_{p'}} \leq [w]_{A_{p'}}^{p'-1}.
$$
Conversely, taking the steps in Figure \ref{fig:pinq} from right to left we have scaled by $\theta = 1-p <0$, giving
$$
[w]_{A_p} \leq [w^{1-p'}]_{A_{p'}}^{p-1}.
$$
Hence,  $[w^{1-p'}]_{A_{p'}}=[w]^{p'-1}_{A_p}$.
\end{proof}

\begin{cor}\label{1inp} Fix $1 < p < \infty$. Then, $w\in A_1$ implies  $w^{1-p}\in A_p \cap RH_\infty$, and we have $\max\{[w^{1-p}]_{A_p}, [w^{1-p}]_{RH_\infty} \} \leq [w]_{A_1}^{p-1}$. Conversely, $w^{1-p}\in A_p \cap RH_\infty$ implies $w \in A_1$, and we have $[w]_{A_1} \leq  ([w^{1-p}]_{A_p} [w^{1-p}]_{RH_\infty} )^\frac{1}{p-1}.$
\end{cor}
\begin{proof} Here is the visual proof
\begin{figure}[H] 
   \centering
   \includegraphics[width=4.7in]{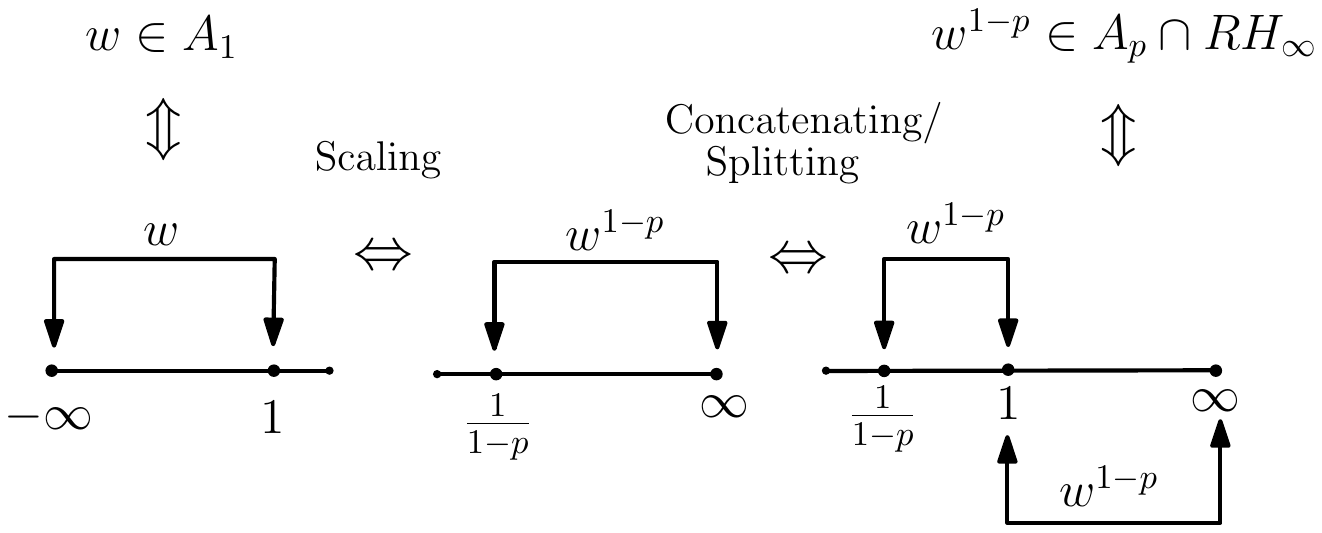}
   \caption{Visual proof of Corollary \ref{1inp}}
   \label{fig:11pp}
\end{figure}
Again, by \thmref{rcmain} and since shrinkage does not worsen reversal constants, for the left-to-right sequence of steps in Figure \ref{fig:11pp}, we have: scaling (with $\theta = 1-p <0$), which yields $[w^{1-p}]_{RC(1/(1-p), \infty)}= [w]_{A_1}^{p-1} $, followed by splitting, which then yields  
$$
\max\{[w^{1-p}]_{A_p}, [w^{1-p}]_{RH_\infty} \} \leq [w]_{A_1}^{p-1}.
$$

On the other hand, now considering the right-to-left sequence of steps, we have: concatenation, which yields $[w^{1-p}]_{RC(1/(1-p), \infty)}\leq [w^{1-p}]_{A_p} [w^{1-p}]_{RH_\infty}$,  followed by scaling (with $\theta = 1/(1-p) <0$), which yields $[w]_{A_1} \leq ([w^{1-p}]_{A_p} [w^{1-p}]_{RH_\infty} )^\frac{1}{p-1}.$
\end{proof}

\begin{cor}\label{1andrhsinrc1minusqinfty} Fix $1 < s < \infty$.
If $w\in A_1\cap RH_s$, then $w^{1-p}\in  A_q \cap RH_\infty$ for every $1< p < \infty$ and $q>(p-1)/s+1$. Moreover, $\max\{[w^{1-p}]_{A_q},[w^{1-p}]_{RH_\infty}\} \leq  ([w]_{A_1}[w]_{RH_s})^{p-1}.$
\end{cor}
\begin{proof} Let us start with the visual proof, which is immediate after realizing that the relation between the indices can be recast as $\frac{s}{1-p} < \frac{1}{1-q}$. We have

\begin{figure}[H] 
   \centering
   \includegraphics[width=5in]{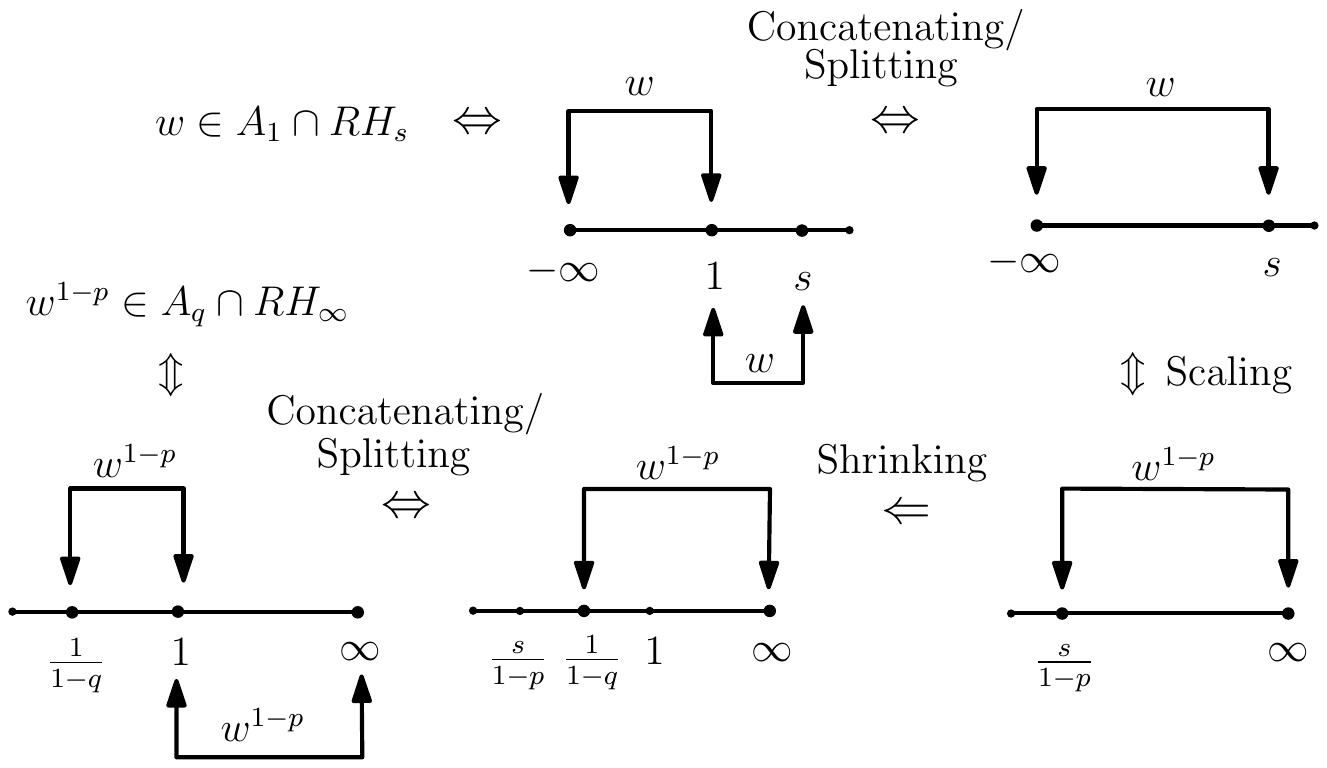}
   \caption{Visual proof of Corollary \ref{1andrhsinrc1minusqinfty}}
   \label{fig:1s1pp}
\end{figure}
For the clockwise sequence of steps in Figure \ref{fig:1s1pp} we have: first concatenation, which yields $[w]_{RC(-\infty,s)}  \leq [w]_{A_1}[w]_{RH_s} $, followed by scaling (with $\theta = 1-p$), yielding $[w^{1-p}]_{RC(s/(1-p), \infty)} \leq ([w]_{A_1}[w]_{RH_s})^{p-1}$, followed by shrinking, which gives $ [w^{1-p}]_{RC(1/(1-q), \infty)} \leq ([w]_{A_1}[w]_{RH_s})^{p-1}$, and finally, splitting gives 
$$
\max\{[w^{1-p}]_{A_q},[w^{1-p}]_{RH_\infty}\}  \leq  ([w]_{A_1}[w]_{RH_s})^{p-1}.
$$
\end{proof}

\begin{cor}\label{rhinfandpin1} Fix $1 < p  < \infty$. Then,  $w\in A_p \cap RH_\infty$ implies that $w^{1-p'}\in A_1$, with $[w^{1-p'}]_{A_1} \leq ([w]_{RH_\infty}[w]_{A_p})^{p'-1}$. Conversely, $w^{1-p'}\in A_1$ implies $w\in  A_p \cap RH_\infty$, with $\max\{[w]_{RH_\infty}, [w]_{A_p}\} \leq [w^{1-p'}]_{A_1}^{p-1}.$
\end{cor}
\begin{proof} Figure \ref{fig:rhinfandpin1} illustrates the visual proof
\begin{figure}[H] 
   \centering
   \includegraphics[width=5in]{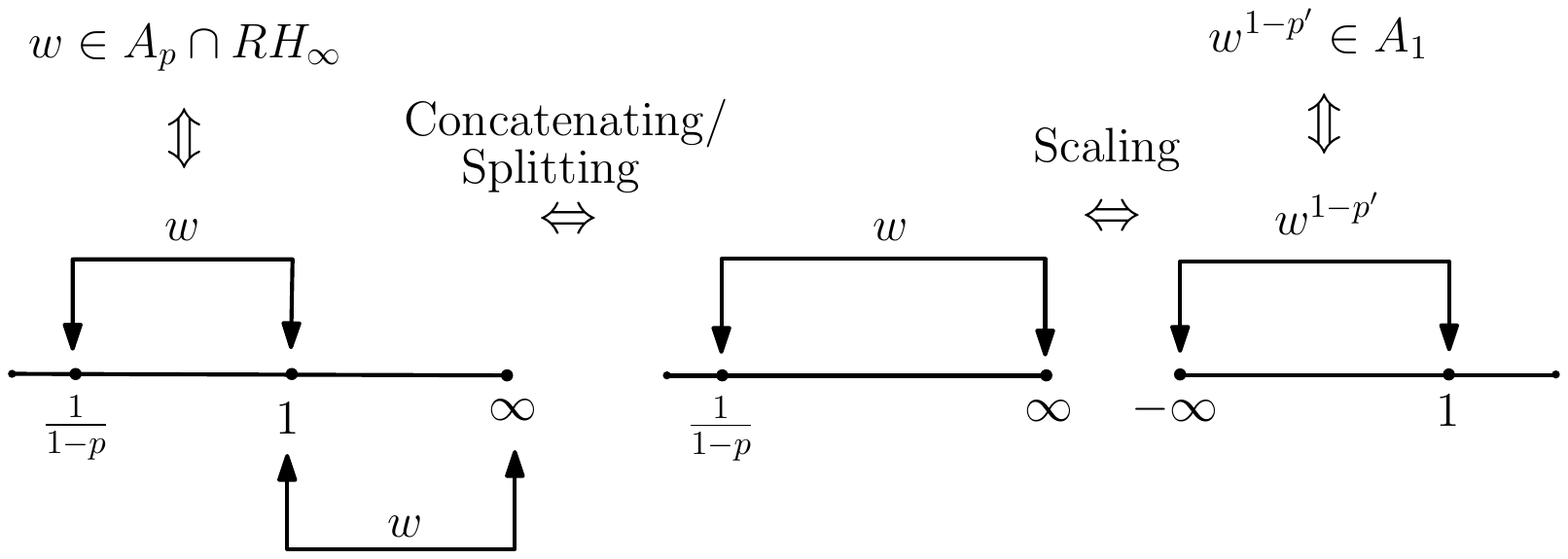}
   \caption{Visual proof of Corollary \ref{rhinfandpin1}}
   \label{fig:rhinfandpin1}
\end{figure}
The proof here closely follows the proof of Corollary \ref{1inp}  along with the fact that $(1-p)(1-p')=1$. For the left-to-right sequence of steps in Figure \ref{fig:rhinfandpin1} we have: concatenation, which yields $[w]_{RH(1/(1-p), \infty)} \leq [w]_{A_p}[w]_{RH_\infty}$ and then scaling by $1-p' < 0$, which gives $[w^{1-p'}]_{A_1} \leq ([w]_{A_p}[w]_{RH_\infty})^{p'-1}$.

For the right-to-left sequence of steps in Figure \ref{fig:rhinfandpin1} we have: scaling by $\frac{1}{1-p'}=1-p$, which yields $[w]_{RH(1/(1-p), \infty)} = [w^{1-p'}]_{A_1}^{p-1}$ and then splitting, which gives
$$
 \max\{[w]_{RH_\infty}, [w]_{A_p}\} \leq [w^{1-p'}]_{A_1}^{p-1}.
$$
\end{proof}

\begin{cor}\label{pandrhsinq} Let $1 < p, s < \infty$ and set  $q:=s(p-1)+1$. Then, $w\in A_p\cap RH_s$ implies $w^s\in A_q,$ with $[w^s]_{A_q} \leq ([w]_{A_p}[w]_{RH_s})^{s}$. Conversely,  $w^s\in A_q$  implies  $w\in A_p\cap RH_s$, with $\max\{ [w]_{A_p},[w]_{RH_s}\} \leq [w^s]_{A_q}^{1/s}.$
\end{cor}
\begin{proof}  Figure \ref{fig:pandrhsinq} illustrates the visual proof
\begin{figure}[H] 
   \centering
   \includegraphics[width=5.2in]{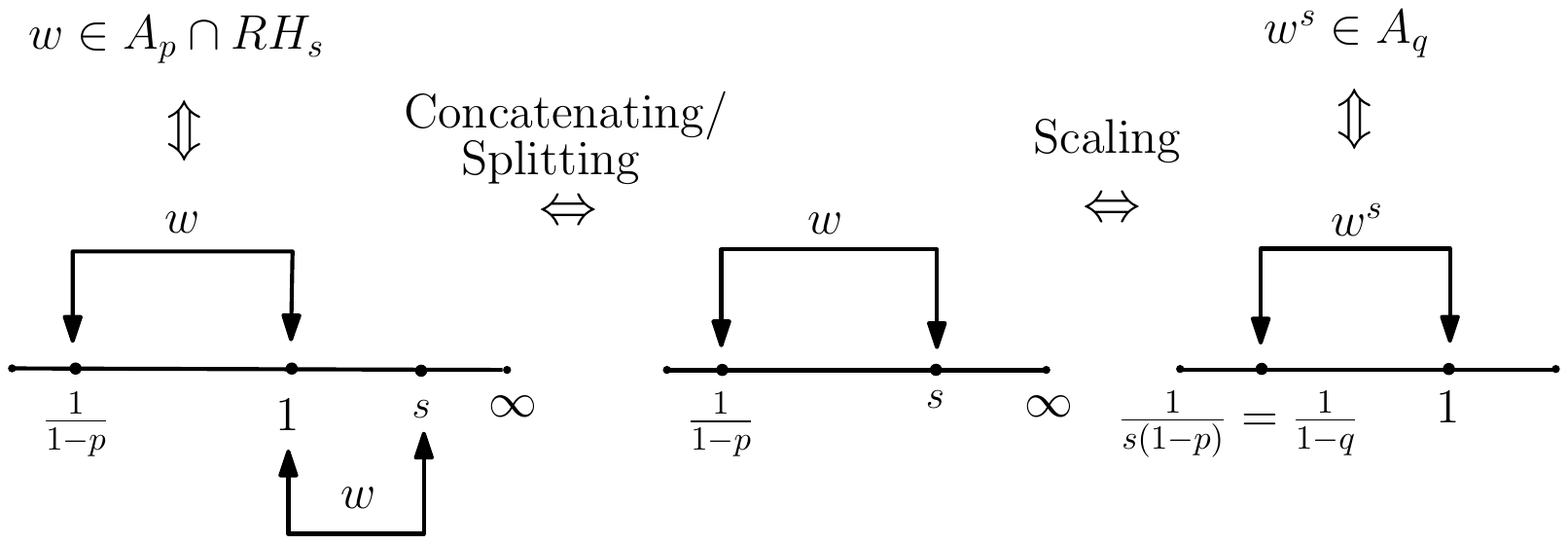}
   \caption{Visual proof of Corollary \ref{pandrhsinq}}
   \label{fig:pandrhsinq}
\end{figure}
In Figure \ref{fig:pandrhsinq}, from left to right, we have: concatenation, which gives\\ $[w]_{RC(1/(1-p), s} \leq [w]_{A_p} [w]_{RH_s}$, followed by scaling by $s$, which gives\\
 $[w^s]_{A_q}  \leq ([w]_{A_p} [w]_{RH_s})^s$. From right to left we have: scaling by $1/s$, yielding $[w]_{RC(1/(1-p),s)} = [w^s]_{A_q}^{1/s}$, followed by splitting, which gives  $\max\{ [w]_{A_p},[w]_{RH_s}\} \leq [w^s]_{A_q}^{1/s}.$
\end{proof}

\section{Self-improving properties}\label{secc:selfimprove}

\begin{dfn} A reverse class $RC(r,s,C_1)$ is said to have the \emph{right self-improving
property} if for every $w\in RC(r,s,C_1)$ there exists $\varepsilon,
C_2>0$ such that $w\in RC(r,s+\varepsilon,C_2).$ 

A reverse class $RC(r,s,C_1)$ is said to have the \emph{left self-improving property}
if for every $w\in RC(r,s,C_1)$ there exists $\eps, C_2>0$ such that
$w\in RC(r-\eps,s,C_2).$ 

These self-improving properties are visually represented in Figure \ref{improverc}.
\begin{figure}[H] 
   \centering
   \includegraphics[width=5.2in]{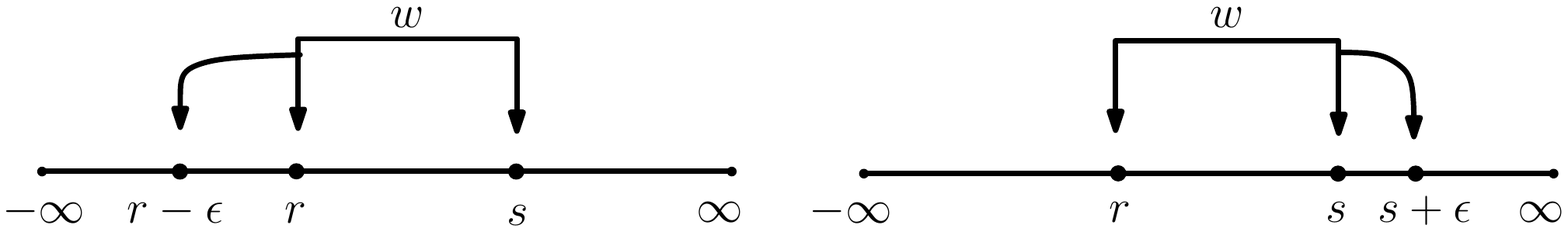}
   \caption{Left and right self-improving properties. The corresponding concatenations are tacitly understood. That is, the diagrams also indicate that $w \in RC(r-\eps, r) \cap RC(r,s) = RC(r-\eps, s)$ and $w \in RC(r,s) \cap RC(s, s+\eps) = RC(r,s+\eps)$, respectively.}
   \label{improverc}
\end{figure}
\end{dfn}

In the next theorem we collect some well-known facts concerning self-improving properties for $A_p$ and reverse-H\"older classes.

\begin{theor}\label{ainftyasap} Fix $1 < p, s < \infty$.
\begin{enumerate}[(i)]
\item\label{rhdelta} For every $w \in A_p$ there exists $\delta > 0$, depending on geometric constants, $p$, and $[w]_{A_p}$, such that $w \in RH_{1+\delta}$.
\item\label{openeps} For every $w \in A_p$ there exists $\eps > 0$, depending on geometric constants, $p$, and $[w]_{A_p}$, such that $w \in A_{p-\eps}$.
\item\label{RHs=>Aq} For every $w \in RH_s$ there exists $1 < q < \infty$, depending on geometric constants, s, and $[w]_{RH_s}$, such that $w \in A_{q}$.
\item\label{Ainf=>Ar} For every $w \in A_\infty$ there exists $1 < r < \infty$, depending on geometric constants and $[w]_{A_\infty}$, such that $w \in A_{r}$.
\end{enumerate}
\end{theor}

Theorem \ref{ainftyasap} illustrates the close connection between the reverse inequalities defining $A_p$ and the reverse H\"older inequalities defining $RH_s$. This connection was first proved by Coifman and Fefferman in \cite{CF74}.  For proofs of Theorem \ref{ainftyasap}, see for instance, \cite{CF74}, \cite[Chapter 7]{Duo}, \cite[Chapter 9]{Grafakos}, and \cite[Chapter 5]{St}. Precise estimates for $\delta > 0$ in \eqref{rhdelta} and $\eps > 0$ in \eqref{openeps}, in the context of spaces of homogeneous type, can be found in \cite{HPR}. Sharp reverse H\"older inequalities for $A_\infty$ weights in $\rn$ can be found in \cite{HP, HPR}.

We can now illustrate the contents of Theorem \ref{ainftyasap} as follows

\begin{figure}[H] 
   \centering
   \includegraphics[width=5in]{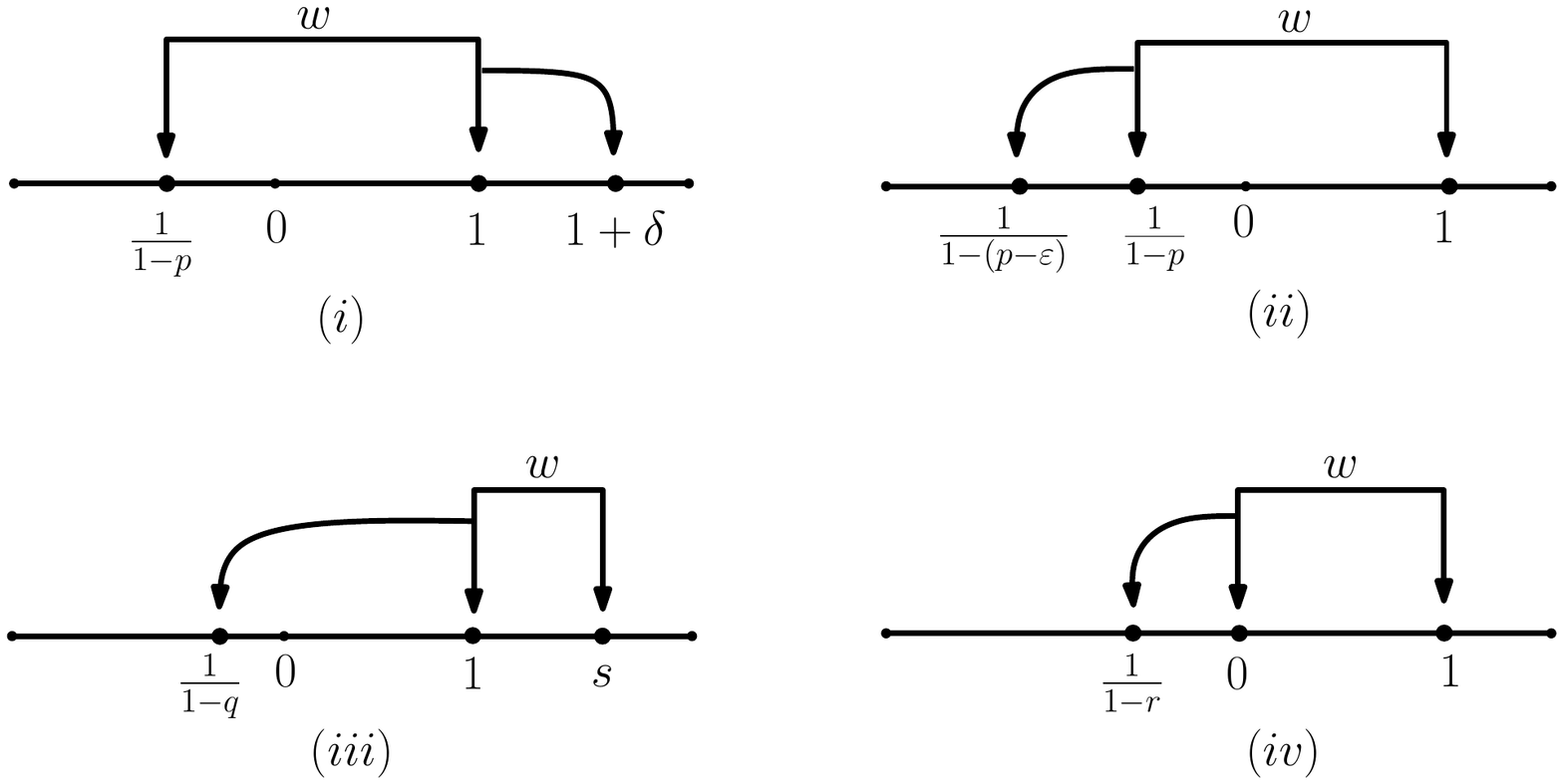}
   \caption{The self-improving properties for the $A_p$ classes as stated in Theorem \ref{ainftyasap}. Notice that the ``improvement leaps'' in parts $(i)$, $(ii)$, and $(iv)$ will typically be small. That is, $\delta$ and $\eps$ will be small and $r$ will be big. However, in the case of $(iii)$, although the index $q$ will also be typically large, the ``improvement leap'' is of length larger than one, crossing from $1$ all the way back to the negative number $\frac{1}{1-q}$.}
   \label{fig:selfimprov}
\end{figure}

In order to keep practicing our visual formalism, let us prove some well-known results using the diagrams for self-improving properties.

\begin{cor}\label{rhsainf} Fix $1 < s < \infty$. If $w \in RH_s$ then $w^s \in A_\infty$.

\end{cor}
\begin{proof} The visual proof is illustrated in Figure \ref{fig:rhsainf}. Notice how the first step uses the self-improving property from Theorem \ref{ainftyasap}  \eqref{RHs=>Aq}.
\begin{figure}[H] 
   \centering
   \includegraphics[width=4.7in]{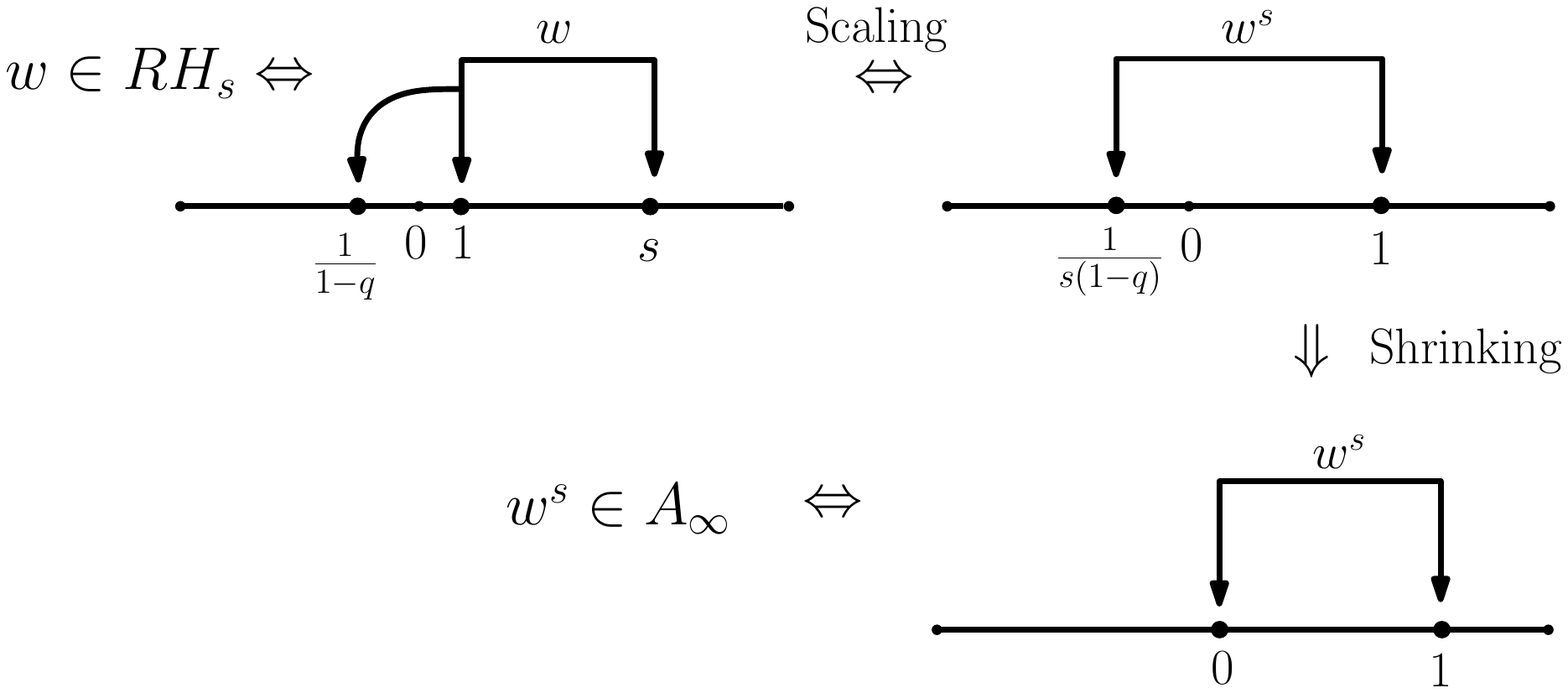}
   \caption{Proof of Corollary \ref{rhsainf}.}
   \label{fig:rhsainf}
\end{figure}
\end{proof}

\begin{cor}\label{rhsrht} Fix $1 < s < \infty$. If $w \in RH_s$, then there exists $t > s$ such that $w \in RC(0,t) \subset RH_t$.

\end{cor}
\begin{proof} The visual proof is illustrated in Figure \ref{fig:rhsrht}. Again, the first step uses the self-improving property from Theorem \ref{ainftyasap}  \eqref{RHs=>Aq}.
\begin{figure}[H] 
   \centering
   \includegraphics[width=5in]{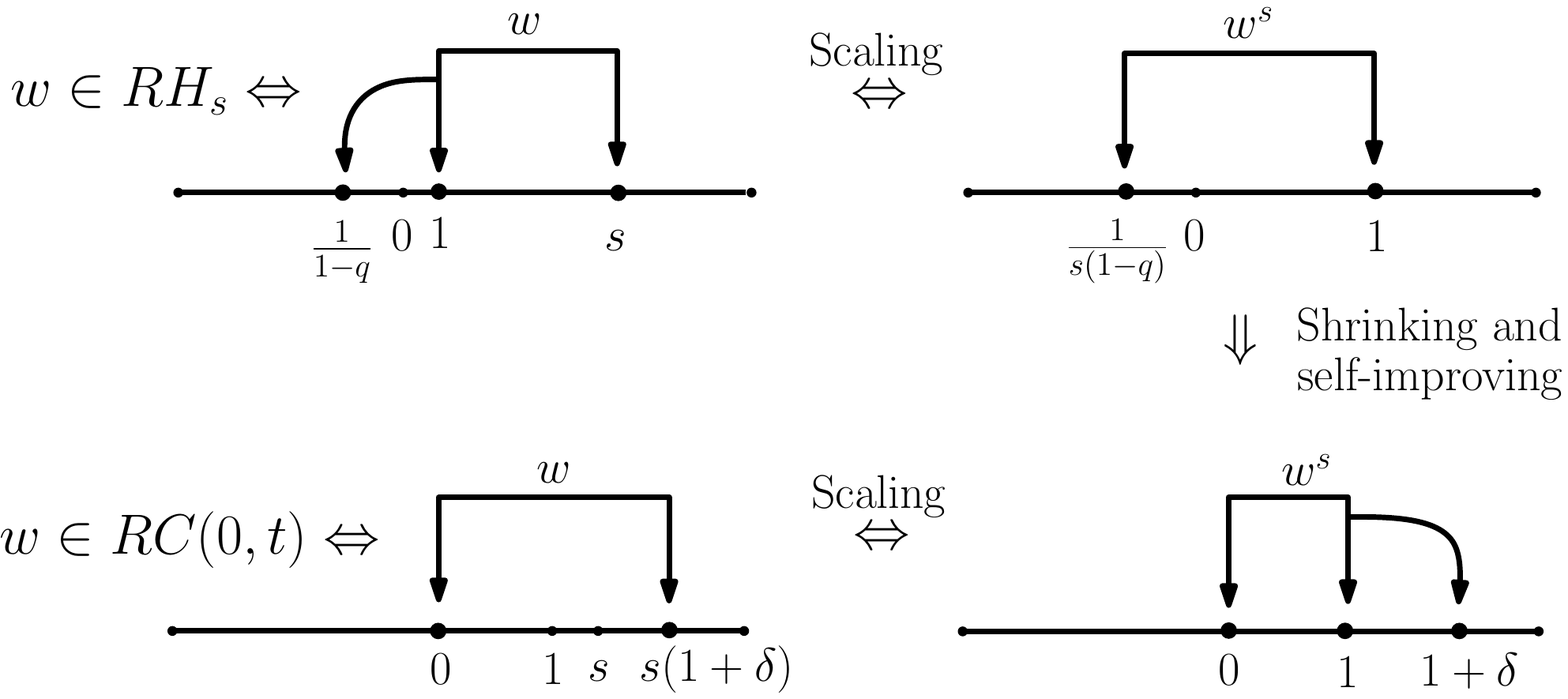}
   \caption{Proof of Corollary \ref{rhsrht} with $t:=s(1+\delta)$.}
   \label{fig:rhsrht}
\end{figure}
\end{proof}

\begin{cor}\label{ainfa2} If $w \in A_\infty$, then there exists $\eps > 0$ such that $w^\eps \in A_2$.
\end{cor}
\begin{proof} The visual proof is illustrated in Figure \ref{fig:ainfa2}. Notice how  the first step uses the self-improving property from Theorem \ref{ainftyasap}  \eqref{Ainf=>Ar}.
\begin{figure}[H] 
   \centering
   \includegraphics[width=5in]{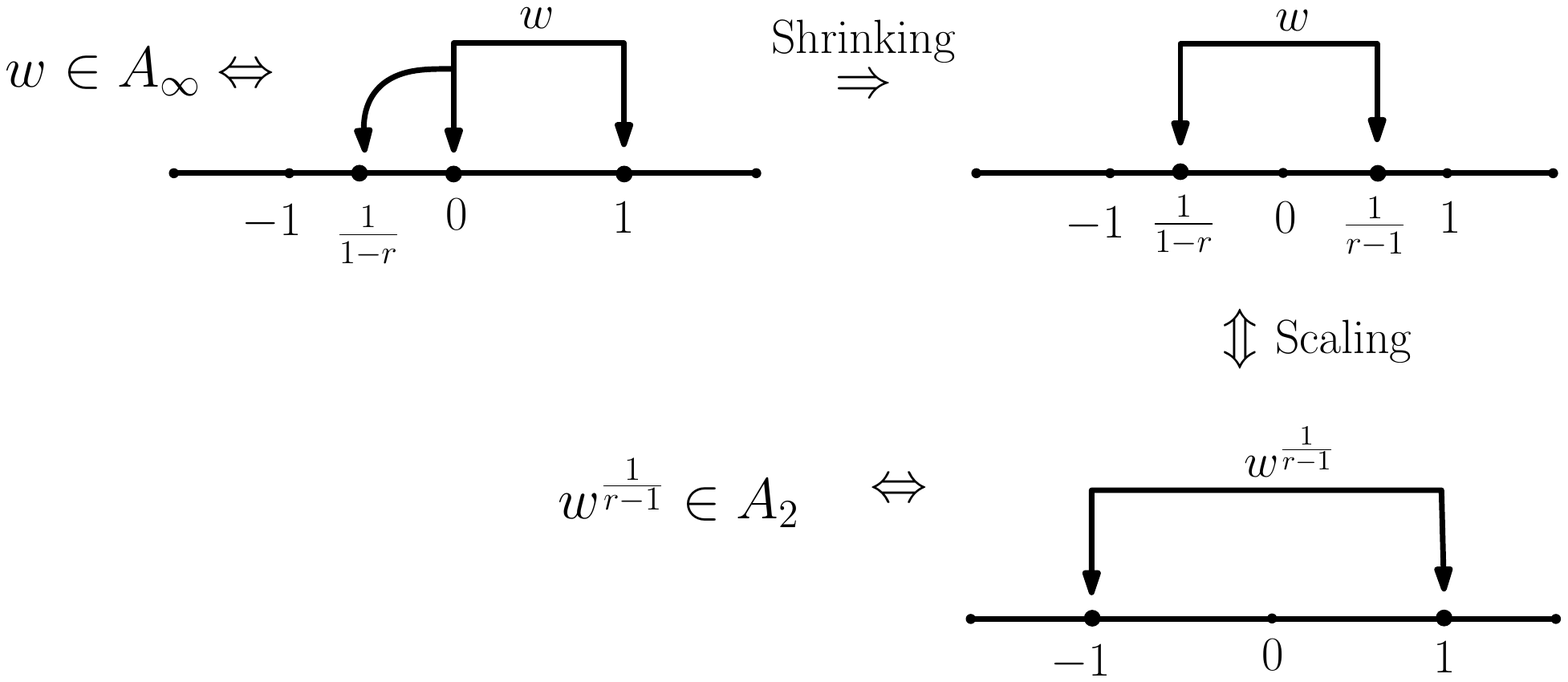}
   \caption{Proof of Corollary \ref{ainfa2} with $\eps := \frac{1}{r-1}$.}
   \label{fig:ainfa2}
\end{figure}
\end{proof}

Weak  reverse classes also enjoy self-improving properties. For instance, we have

\begin{theor}\label{selfWRH} Fix $1 < s < \infty$ and  $0 < p < q < r \leq \infty$. Then,
\begin{enumerate}[(i)]
\item\label{selfWRHLHS}  for every $w \in RH^{weak}_s$ there exists $\eps > 0$ such that $w \in RH^{weak}_{s+\eps}$, and
\item\label{selfWRHpq} $RC^{weak}(q,r) = RC^{weak}(p,r)$.
\end{enumerate}
Visually, these properties are depicted in Figures \ref{fig:selfimproveRHweakRHS} and \ref{fig:selfimproveWRH}.
\begin{figure}[H] 
   \centering
   \includegraphics[width=4in]{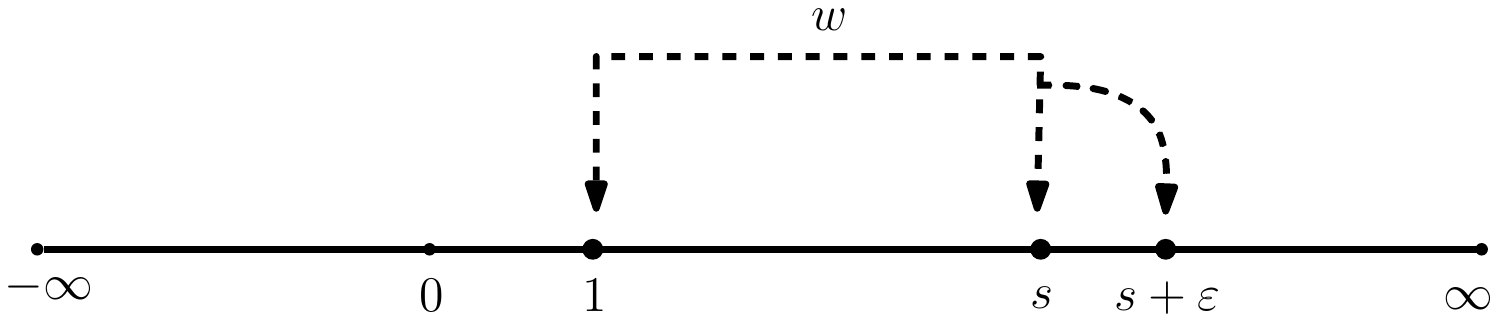}
   \caption{The self-improving property for the weak reverse class $RH^{weak}_s$ from Theorem \ref{selfWRH} (\ref{selfWRHLHS}).}
   \label{fig:selfimproveRHweakRHS}
\end{figure}

\begin{figure}[H] 
   \centering
   \includegraphics[width=4.6in]{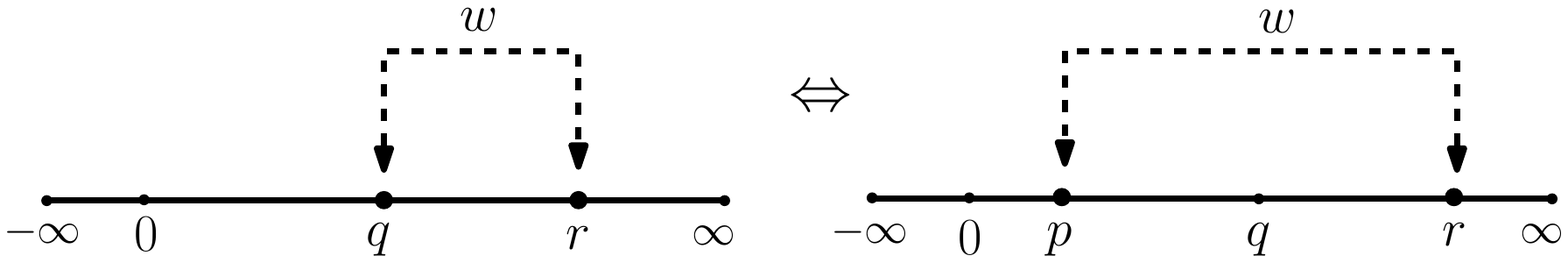}
   \caption{The self-improving property for $RC^{weak}(q,r)$ from Theorem \ref{selfWRH} (\ref{selfWRHpq}). }
   \label{fig:selfimproveWRH}
\end{figure}

\end{theor}

The whole phenomenon of reverse inequalities and their self-improving properties stems from Gehring's work \cite{Ge74} (see \cite[Chapter 3]{BB11} for an exposition of the so-called Gehring's lemma in doubling metric spaces). As mentioned, the connection between reverse H\"older inequalities and $A_p$ weights was first explored in \cite{CF74}. A proof for Theorem \ref{selfWRH} (\ref{selfWRHLHS}) in spaces of homogeneous type can be found, for instance, in \cite{Kin}. A proof for Theorem \ref{selfWRH} (\ref{selfWRHpq}) in spaces of homogeneous type can be found, for instance, in \cite[Lemma 1.4]{BKL} and the case $r=\infty$ has been worked out, for instance, in \cite[Remark 4.4]{KS}.

\section{Reverse classes, BMO, BLO, and BUO}

Throughout this section we will consider $\Omega = X$.

\begin{dfn}\label{bmo} We recall the definitions of $BMO$, $BLO$ and $BUO$; namely, the spaces of functions of bounded mean oscillation, bounded lower oscillation, and bounded upper oscillation, respectively. $w\in BMO$  if and only if 
\bee\label{bmo}
\|w\|_{BMO}:=\ds\sup_{B\in\B}
\left(\frac{1}{\mu(B)}\int_B \left|w-w_B\right| \,d\mu\right) <
\infty,
\eee
where, as usual,  $w_B:=\left(\frac{1}{\mu(B)}\int_B w \,d\mu\right)$. $w\in BLO$  if and only if
\bee\label{blo}
\|w\|_{BLO}:=\ds\sup_{B\in\B}
\left(\frac{1}{\mu(B)}\int_B (w-\essinf_B w)
\,d\mu\right) < \infty. 
\eee
$w\in BUO$  if and only if
\bee\label{buo}
\|w\|_{BUO}:=\ds\sup_{B\in\B}
\left(\frac{1}{\mu(B)}\int_B (\esssup_B w - w)
\,d\mu\right) < \infty. 
\eee
\end{dfn}

\propref{bmobloap} below provides characterizations of $BMO$
and $BLO$ in terms of $A_p$ classes.

\begin{prop}\label{bmobloap}
The following characterizations hold true:
\begin{enumerate}[$(i)$]
\item\label{bmoap}
$\log w\in BMO$ if and only if $w^\alpha \in A_p$ for some $1 < p < \infty$ and some
$\alpha \in \re$.
\item\label{bloap}
$\log w\in BLO$ if and only if $w^\eps \in A_1$ for some $\eps>0$.
\item\label{buoap}
$\log w\in BUO$ if and only if $w^\eps \in RH_\infty$ for some $\eps>0$.
\end{enumerate}
\end{prop}
For a proof of \eqref{bmoap} see, for instance, \cite[p.151]{Duo}, \cite[p.300]{Grafakos}, and \cite[p.218]{St}. The proof of \eqref{bloap} can be found in \cite[Lemma 1]{CR80}. The proof of \eqref{buoap} follows, for instance, from the characterization $RH_\infty = e^{BUO}$ as in \cite[Theorem 3.2]{Ou08}.
In view of Proposition \ref{bmobloap}, we have

\begin{cor}\label{characbmoblo} The following characterizations hold true:
\begin{enumerate}[$(i)$]
\item\label{charbmo} $\log w\in BMO$ if and only if $w\in RC(r,s)$ for some $-\infty \leq r < s \leq \infty$.
\item\label{charblo} $\log w\in BLO$ if and only if $w\in RC(-\infty,r)$ for some $-\infty < r \leq \infty$.
\item\label{charbuo} $\log w\in BUO$ if and only if $w\in RC(r,\infty)$ for some $-\infty \leq r < \infty$.
\end{enumerate}
Visually,
\begin{figure}[H] 
   \centering
   \includegraphics[width=4.5in]{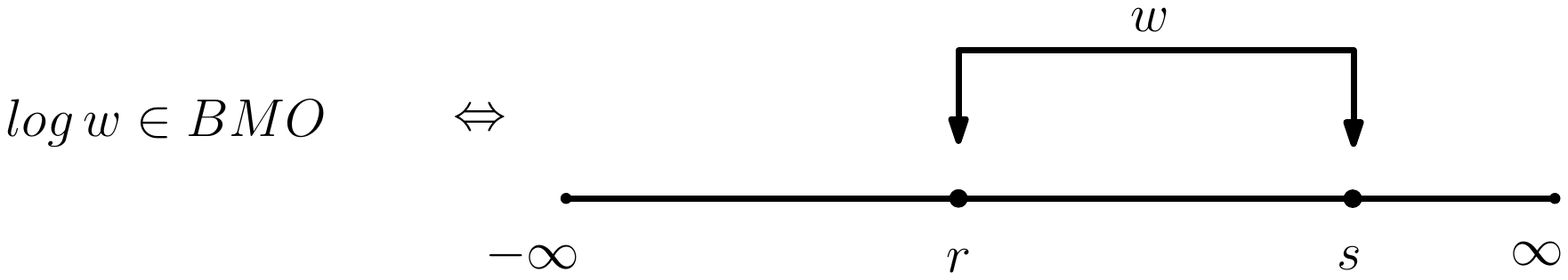}
   \caption{A characterization of $BMO$: $BMO = \log \left( \bigcup\limits_{-\infty \leq r <  s \leq \infty} RC(r,s)\right)$.}
   \label{bmo}
\end{figure}

\begin{figure}[H] 
   \centering
   \includegraphics[width=4.5in]{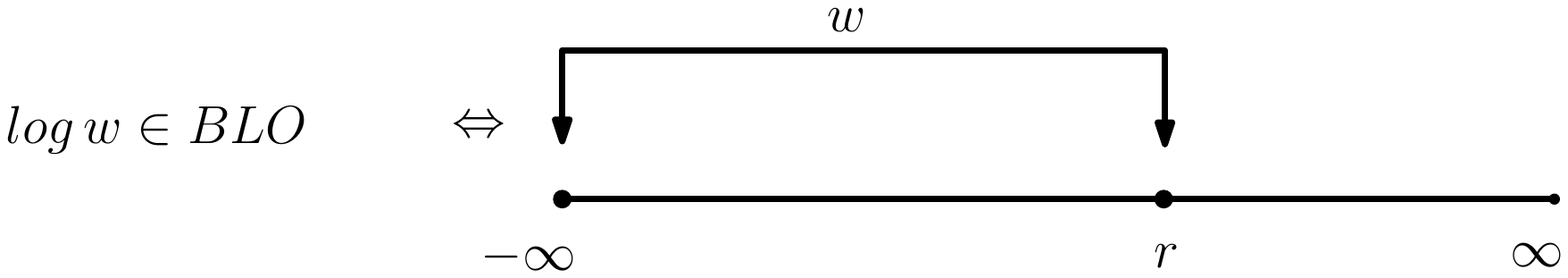}
   \caption{A characterization of $BLO$: $BLO = \log \left( \bigcup\limits_{-\infty < r \leq \infty} RC(-\infty,r)\right)$.}
   \label{blo}
\end{figure}

\begin{figure}[H] 
   \centering
   \includegraphics[width=4.5in]{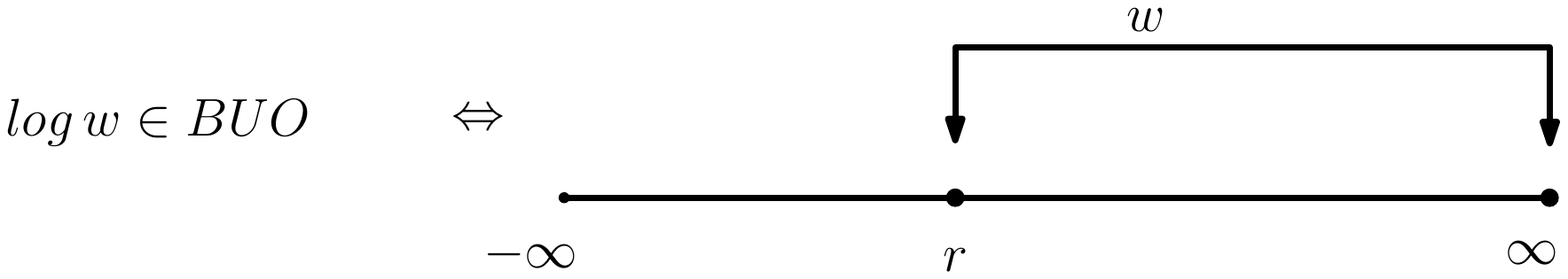}
   \caption{A characterization of $BUO$: $BUO = \log \left( \bigcup\limits_{-\infty \leq r < \infty} RC(r, \infty)\right)$.}
   \label{blo}
\end{figure}

\end{cor}

\begin{proof} Here is the visual proof of Corollary \ref{characbmoblo} by means of Proposition \ref{bmobloap}.

\begin{figure}[H] 
   \centering
   \includegraphics[width=5in]{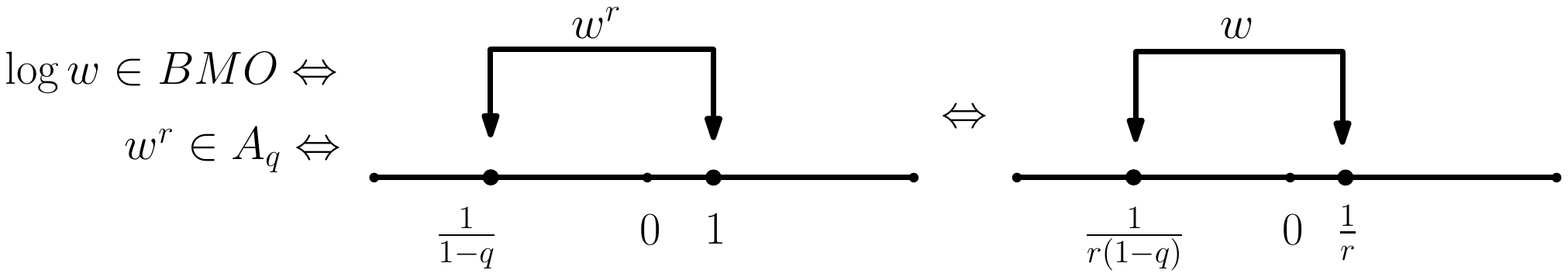}
   \caption{Proof of the ``only if'' part of  Corollary \ref{characbmoblo} (\ref{charbmo}), for $0 < r $, using Proposition \ref{bmobloap} (\ref{bmoap}). In the case of $r<0$, only the positions of the points $1/r$ and $1/(r(1-q))$ will be switched in the figure.}
   \label{fig:bmoimpliesrc}
\end{figure}

\begin{figure}[H] 
   \centering
   \includegraphics[width=5in]{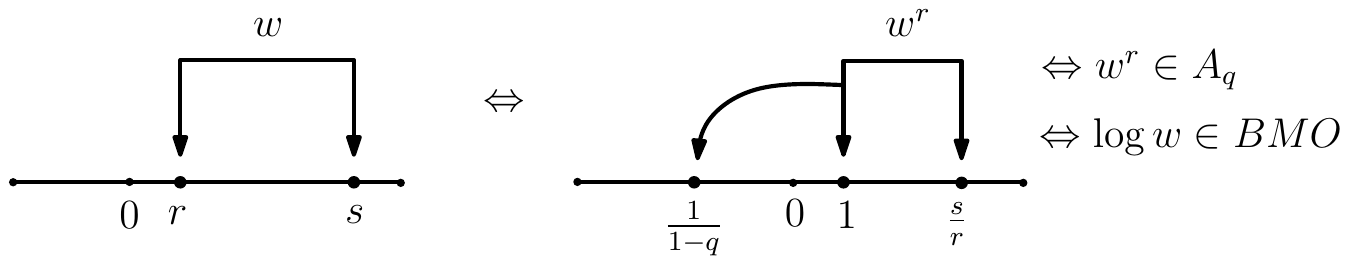}
   \caption{Proof of the ``if'' part of Corollary \ref{characbmoblo} (\ref{charbmo}), for $0 < r < s $, by using Proposition \ref{bmobloap} (\ref{bmoap}). The case $r < s < 0$ reduces to this case after scaling by $-1$. By the shrinking property, the general case $r < s$ reduces to one of the previous two cases (see Exercise \ref{RCcrosszero} in Section \ref{secc:ex}).}
   \label{fig:bmorspositive}
\end{figure}

\begin{figure}[H] 
   \centering
   \includegraphics[width=4.6in]{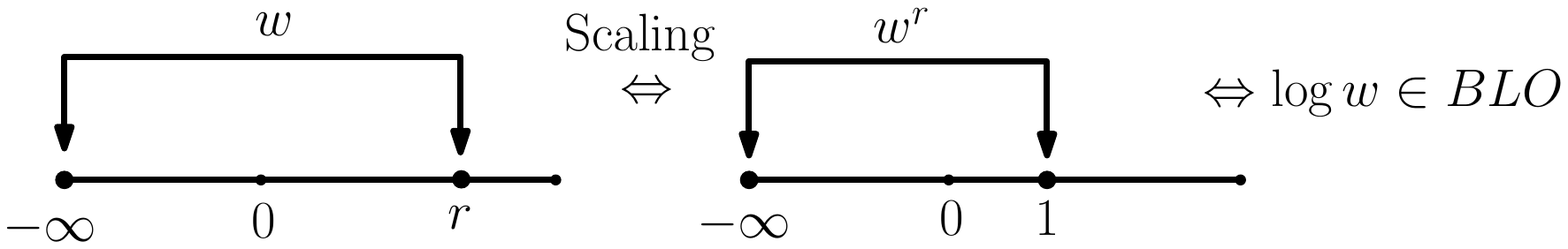}
   \caption{Proof of Corollary \ref{characbmoblo} (\ref{charblo}) for the case $r> 0$, by using Proposition \ref{bmobloap} (\ref{bloap}).}
   \label{fig:blorpositive}
\end{figure}

\begin{figure}[H] 
   \centering
   \includegraphics[width=5in]{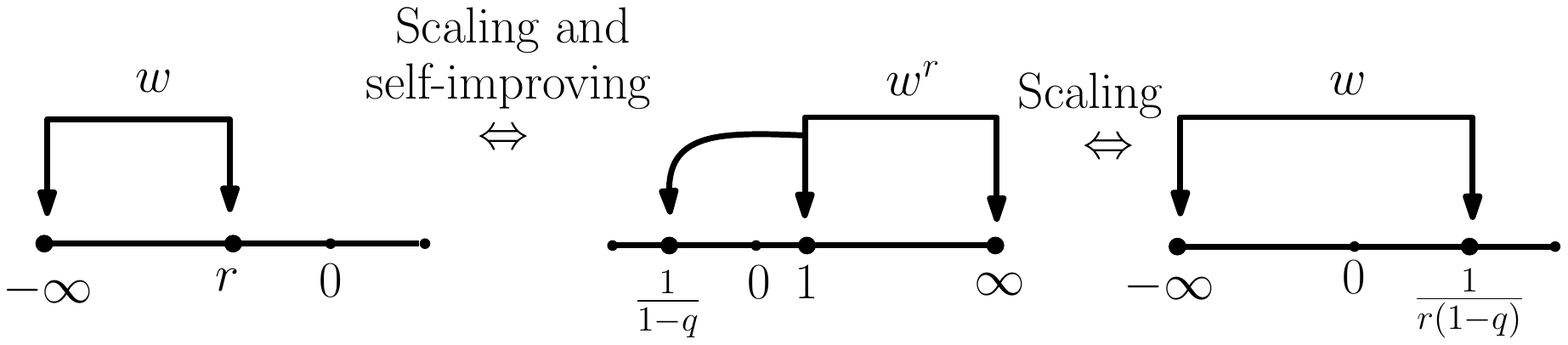}
   \caption{Proof of Corollary \ref{characbmoblo} (\ref{charblo}) for the case $r < 0$. It is shown how after scaling and self-improving, this case reduces to the case $r > 0$. }
   \label{fig:blornegative}
\end{figure}
The proof of Corollary \ref{characbmoblo} \eqref{charbuo} is left as an exercise (see Exercise \ref{exe:charabuo} in Section \ref{secc:ex}).

\end{proof}

\section{Interpolation}\label{secc:interpolation}

\begin{theor}\label{interpolation}
Let $r_1\leq r_2<0<s_1\leq s_2$, and $u\in RC(r_1,s_1,C_1)$ and
$v\in RC(r_2,s_2,C_2)$. Then, for every $\theta\in[0,1]$, we have $u^\theta v^{1-\theta}\in RC(r_\theta,s_\theta,C_\theta)$, where $C_\theta := C_1^\theta C_2^{1-\theta} $ and
$$
\frac{1}{r_\theta}:=\frac{\theta}{r_1} + \frac{1-\theta}{r_2} < 0, \quad  \quad \frac{1}{s_\theta}:=\frac{\theta}{s_1} + \frac{1-\theta}{s_2} > 0.
$$
Visually, this is illustrated in Figure \ref{fig:rcinterpol}.
\begin{figure}[H] 
   \centering
   \includegraphics[width=5.2in]{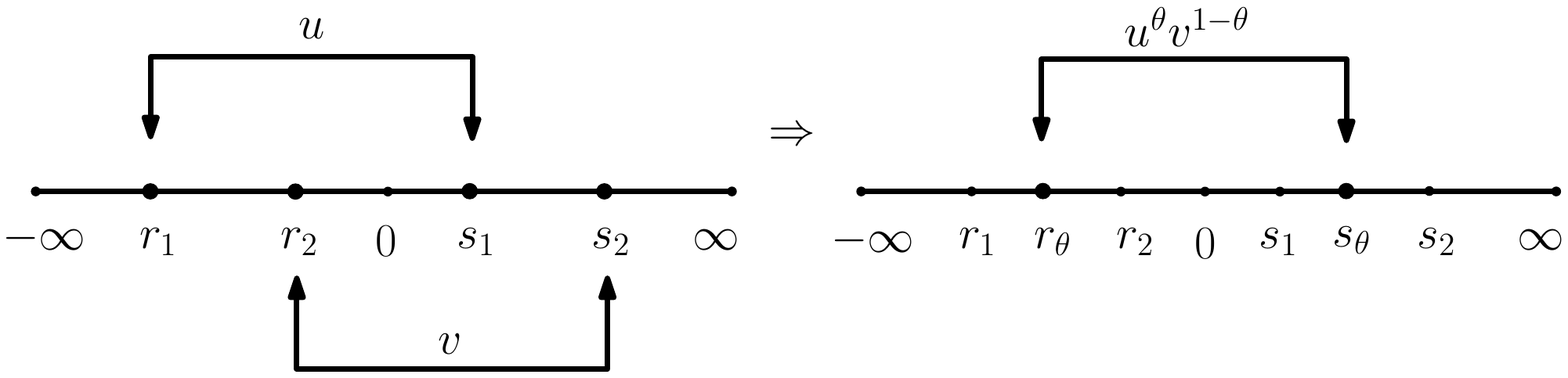}
   \caption{Interpolation of two reverse classes crossing zero.}
   \label{fig:rcinterpol}
\end{figure}
\end{theor}

\begin{proof} Assume $\theta\in(0,1)$ (the cases $\theta=0$ or $\theta=1$ being trivial).
We need to prove that 
\be\label{tmf1} \left(\frac{1}{\mu(B)}\int_B u^{s_\theta\theta} v^{s_\theta(1-\theta)}
\,d\mu\right)^{1/s_\theta}\left( \frac{1}{\mu(B)}\int_B
u^{r_\theta\theta}v^{(1-\theta)r_\theta}\,d\mu\right)^{-1/r_\theta}\leq
C_\theta. 
\ee
By H\"{o}lder's inequality with the exponents $q:=\frac{r_1}{\theta r_\theta}$ and $q':=\frac{r_2}{(1-\theta)r_\theta}$, one gets (recall the notation in \eqref{pmeanB})
\be\begin{split}\label{tf1} 
\left( \frac{1}{\mu(B)}\int_B
u^{r_\theta\theta}v^{(1-\theta)r_\theta}\,d\mu\right)^{-1/r_\theta}&\leq u(r_1,B)^{-\theta} v(r_2,B)^{-(1-\theta)}.
\end{split}\ee 
Again, applying H\"{o}lder's inequality with the exponents $r:=\frac{s_1}{\theta s_\theta}$ and $r':=\frac{s_2}{(1-\theta)s_\theta}$, and using the assumptions $u\in
RC(r_1,s_1,C_1)$ and $v\in RC(r_2,s_2,C_2)$,
\be\begin{split}\label{tf2} \left(\frac{1}{\mu(B)}\int_B
u^{s_\theta\theta} v^{s_\theta(1-\theta)}\,d\mu\right)^{1/s_\theta}
&\leq w(s_1,B)^\theta w(s_2, B)^{1-\theta}\\&\leq C_1^\theta C_2^{1-\theta} u(r_1,B)^{\theta}v(r_2,B)^{1-\theta}.
\end{split}\ee 
Therefore, \eqref{tmf1} follows from \eqref{tf1} and
\eqref{tf2}.
\end{proof}

\begin{cor}\label{apinterpolation}
Let $w_j\in A_{p_j}$, $j=1,2,$ where $1\leq p_1<p_2<\infty$, then
$w_1^\theta w_2^{1-\theta}\in A_{p_3}$ where $p_3=\theta
p_1+(1-\theta)p_2$ and $\theta\in(0,1)$.
\end{cor}
\begin{proof}
The proof follows immediately from \thmref{interpolation} after the
statement is recast in terms of reverse classes as follows: if
$w_j\in RC(\frac{1}{1-p_j},1,C_j)$, where $j=1,2,$ and $1\leq
p_1<p_2<\infty$, then $w_1^\theta w_2^{1-\theta}\in
RC(\frac{1}{1-[\theta p_1+(1-\theta)p_2]},1,C_\theta),\,\theta\in(0,1)$.
\end{proof}

\section{Factorization}\label{secc:factorization}

The basic factorization theorem for $A_p$ weights reads 

\begin{theor}\label{apfactorization} $w\in A_p$ if and only if
$w=w_1w_2^{1-p}$, for some $w_1, w_2\in A_1$. 
\end{theor}

Theorem \ref{apfactorization} was first proved, in the Euclidean setting, by P. Jones in \cite{Jones80} and other proofs then appeared in \cite{CJRdF, RdeF84}. For a proof in spaces of homogeneous type, see \cite[Chapter II]{StrTor}. A thorough discussion of this topic can be found, for instance, in   \cite[Sections 1.1 and 6.2]{CMP}, \cite[Section 5.5]{Duo}, and \cite[Chap IX]{Tor}.  

The visual rendition of the factorization theorem yields a certain intersection property for the arrows, see Figure \ref{fig:factorviaa1}. 

\begin{figure}[H] 
   \centering
   \includegraphics[width=5.2in]{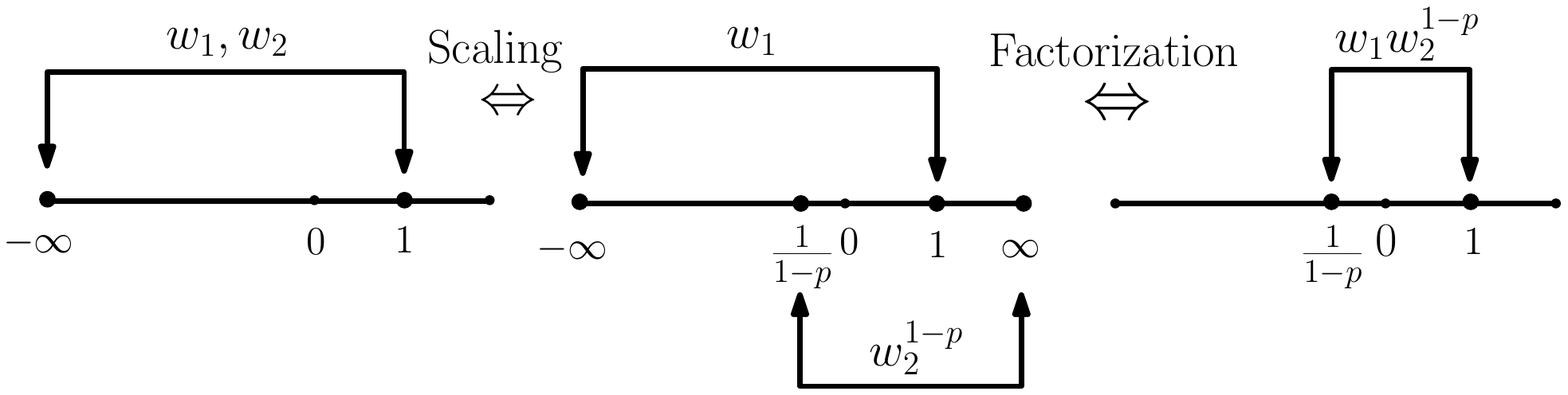}
   \caption{The factorization of $A_p$ weights as an intersection property for the arrows.}
   \label{fig:factorviaa1}
\end{figure}
 
Theorem \ref{apfactorization} can be extended as follows:

\begin{theor}\label{factorization}
Let $r<0<s$. Then $w\in RC(r,s)$ if and only if $w=uv$ for some
$u\in RC(r,\infty)$ and $v\in RC(-\infty,s)$. Visually,
\begin{figure}[H] 
   \centering
   \includegraphics[width=5.3in]{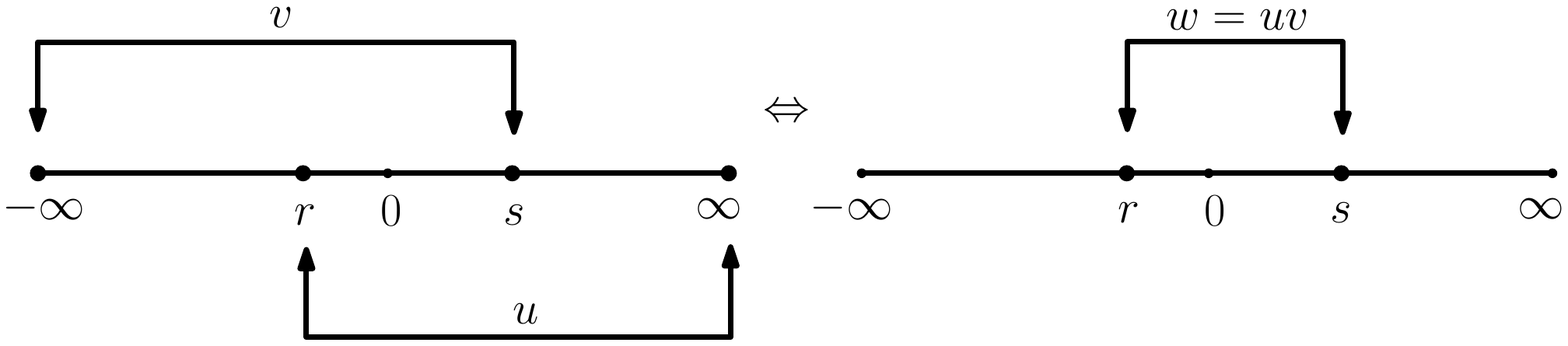}
   \caption{The factorization of a weight $w$ belonging to a reverse class crossing zero as an intersection property for the arrows.}
   \label{product}
\end{figure}
\end{theor}

\begin{proof} For the `if' part, we take $u\in RC(r,\infty,C_2)$
and $v\in RC(-\infty,s,C_1)$ and we will prove $uv\in
RC(r,s,C_1 C_2)$, that is, 
\be\label{mtf1}
(uv)(s,B) \leq C_1 C_2 \, (uv)(r,B)\quad \forall B \in \B.
 \ee
Using the assumptions $u\in RC(r,\infty,C_2)$ and $v\in
RC(-\infty,s,C_1)$, given $B \in \B$ we have 
\begin{equation}\label{mtf2}
(uv)(s,B) \leq ( \esssup_B u) v(s,B) \leq C_2 \,  u(r,B) v(s,B)
\end{equation}
and
\begin{equation}\label{mtf3} 
(uv)(r,B)^{-1} \leq (\essinf_B v)^{-1} u(r,B)^{-1}\leq C_1 \, v(s,B)^{-1} u(r,B)^{-1}.
\end{equation}
Hence, the estimate \eqref{mtf1} follows by multiplying \eqref{mtf2} and \eqref{mtf3}. 

For the `only if' part, we take $w\in RC(r,s)$ and need to
show that there exist $u\in RC(r,\infty)$ and $v\in
RC(-\infty,s)$ such that $w=uv$. We know that $w \in RC(r,s)$, then, by  \thmref{rcmain} (scaling), we have $w^s \in A_p$, with  $p:=(1-\frac{s}{r}) > 1$. Then, by Theorem \ref{apfactorization} (factorization), there exist
$\tilde{u}, \tilde{v}\in A_1$ such that $w^s =
\tilde{u}\tilde{v}^{1-p}$. Defining $u:=\tilde{u}^{\frac{1}{s}}$ and
$v:=\tilde{v}^{\frac{1-p}{s}}$, we get $w^s=uv$. Now, since
$\tilde{u},\tilde{v}\in A_1= RC(-\infty,1)$ and $s > 0$, by \thmref{rcmain} (scaling), we have that $u:=\tilde{u}^{\frac{1}{s}}\in RC(-\infty,s)$. On the other hand, since $r=\frac{s}{1-p}$, from the definition of $p$, by scaling we have that $v:=\tilde{v}^{\frac{1-p}{s}}\in
RC(\frac{s}{1-p},\infty)=RC(r,\infty).$ 
\end{proof}

\begin{cor}{\rm{(See \cite[Section III]{CR80})}}\label{bmoblo} Every $BMO$ function is the difference of two $BLO$ functions, that is,
$$
BMO=BLO-BLO.
$$
\end{cor}
\begin{proof}
In light of \corref{characbmoblo} \eqref{charbmo} and \eqref{charblo}, we only need to show that $w\in
RC(r,s)$, for some  $r<0<s$, if and only if $w=w_1w_2^{-1}$ for some
$w_1\in RC(-\infty,r_1)$ and $w_2\in RC(-\infty,r_2)$, with  $r_1, r_2>0$. Let
$r<0<s$ and $w\in RC(r,s)$. Then, take $w_1\in RC(-\infty,s)$ and
$w_2\in RC(-\infty,-r)$ provided by \thmref{factorization} to get
$w=w_1w_2^{-1}$. Conversely, let $w_1\in RC(-\infty,r_1)$ and
$w_2\in RC(-\infty,r_2)$, for some $r_1, r_2>0$. Finally, since $-r_2<0<r_1$, by \thmref{factorization} we have
$w_1w_2^{-1}\in RC(-r_2,r_1)$.
\end{proof}

\subsection{The visual formalism as a tool for generating conjectures}  Figure \ref{product}  motivates the following question:

\begin{figure}[H] 
   \centering
   \includegraphics[width=4.7in]{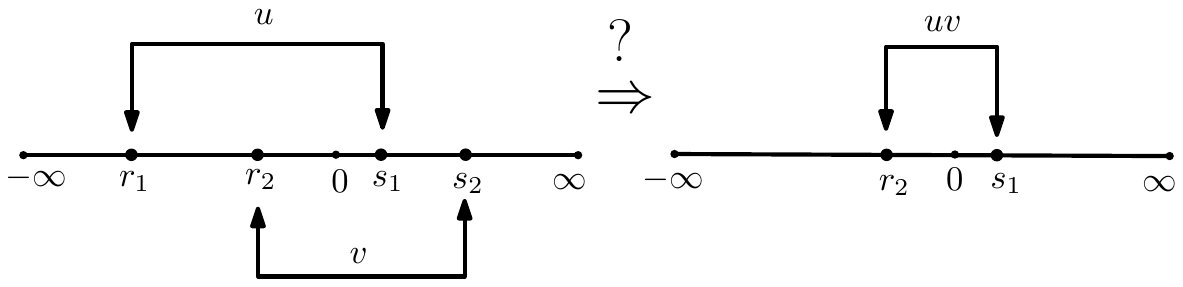}
   \caption{Is the implication true? Does the product honor the reversal property in the intersection of the reversal indices of the factors? It turns out that the answer is no, unless $r_1 = -\infty$ and $s_2 = \infty$ (and then Theorem \ref{factorization} applies).}
   \label{fig:questionproduct}
\end{figure}

Next, we will see that in order for the implication in Figure \ref{fig:questionproduct} to hold true it is necessary that both factors touch opposite infinities. This fact will be substantiated by Examples \ref{0infty} and \ref{1infty} below. These examples are constructed
with the help of the following fact:
$|x|^a\in A_p(\Rn)$ if and only if $-n<a<n(p-1)$, see \cite[p.286]{Grafakos}; visually, this fact is represented in Figure \ref{fig:powersinap}.

\begin{figure}[H] 
   \centering
   \includegraphics[width=4in]{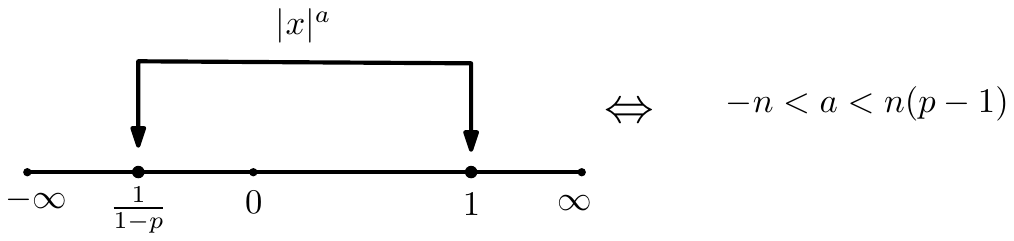}
   \caption{A necessary and sufficient condition for the power weight $|x|^a$ to be in $A_p(\rn)$.}
   \label{fig:powersinap}
\end{figure}

\begin{example}\label{0infty} Take $\Omega = \R$. In this example the reverse classes of both the factors do not touch
an infinity and we will see that
the product does not lie in the intersecting reverse class. Take
$u:=|x|\in A_3=RC(-\frac{1}{2},1)$ and $v:=|x|^3\in
A_5=RC(-\frac{1}{4},1)$. Clearly, we have
\[RC\left(-\frac{1}{2},1\right)\cap RC\left(-\frac{1}{4},1\right)=RC\left(-\frac{1}{4},1\right).\]
But, $uv=|x|^4\notin  A_5$. This example is illustrated in
\figref{noinftyproduct}.

\begin{figure}[H] 
   \centering
   \includegraphics[width=5in]{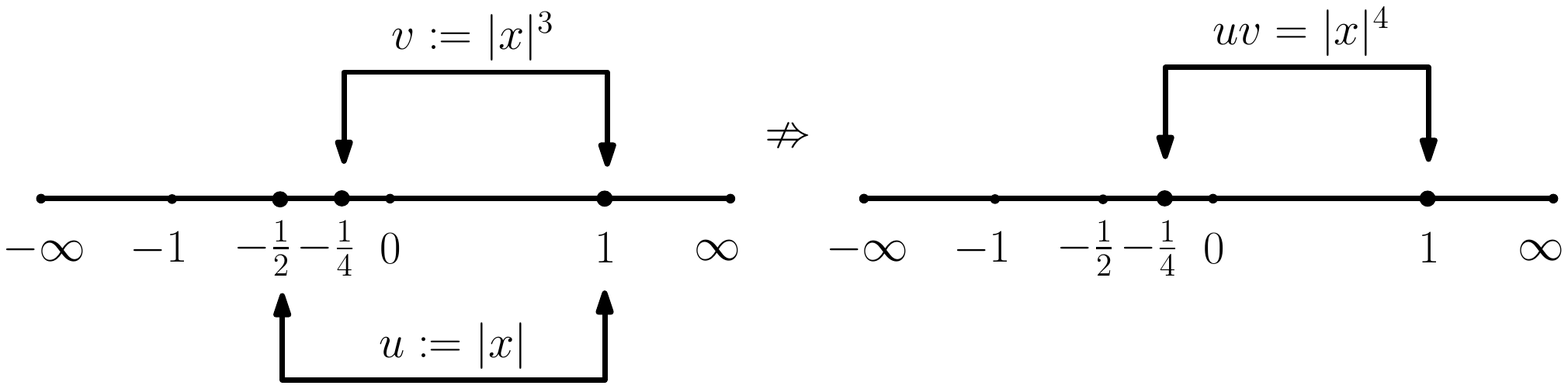}
  \caption{The product will not be in the intersecting reverse class if the reverse classes for
   both factors do not touch an infinity.}
   \label{noinftyproduct}
\end{figure}

\end{example}

\begin{example}\label{1infty} Take $\Omega = \R$. 
In the following example, illustrated in \figref{oneinftyproduct},
only one of the factors touches an infinity and we will see that this is not sufficient for the product to lie in
the intersecting reverse class. Take $\eps_0,\eps_1>0$ so that
$0<(\eps_0+\eps_1)\ll 1$. Then, since $-1<-1+\eps_0<0$,
$|x|^{-1+\eps_0}\in A_1=RC(-\infty,1)$. Equivalently,
$u:=|x|^{1-\eps_0}\in RC(-1,\infty)$. Again, since
$-1<\eps_0+\frac{\eps_1}{2}<\eps_0+\eps_1$,
$v:=|x|^{\eps_0+\frac{\eps_1}{2}}\in
A_{1+\eps_0+\eps_1}(\R)=RC(\frac{1}{-(\eps_0+\eps_1)},1)$. Now,
since $\frac{1}{-(\eps_0+\eps_1)}\ll-1$,
$RC(\frac{1}{-(\eps_0+\eps_1)},1)\subset\subset RC(-1,1)$. Hence,
$v:=|x|^{\eps_0+\frac{\eps_1}{2}}\in RC(-1,1)=A_2$. Clearly,
\[RC(-1,\infty)\cap RC(-1,1)=RC(-1,1).\] But,
$uv=|x|^{1+\frac{\eps_1}{2}}\notin A_2$ since
$1+\frac{\eps_1}{2}>1$.

\begin{figure}[H]
   \centering
   \includegraphics[width=5in]{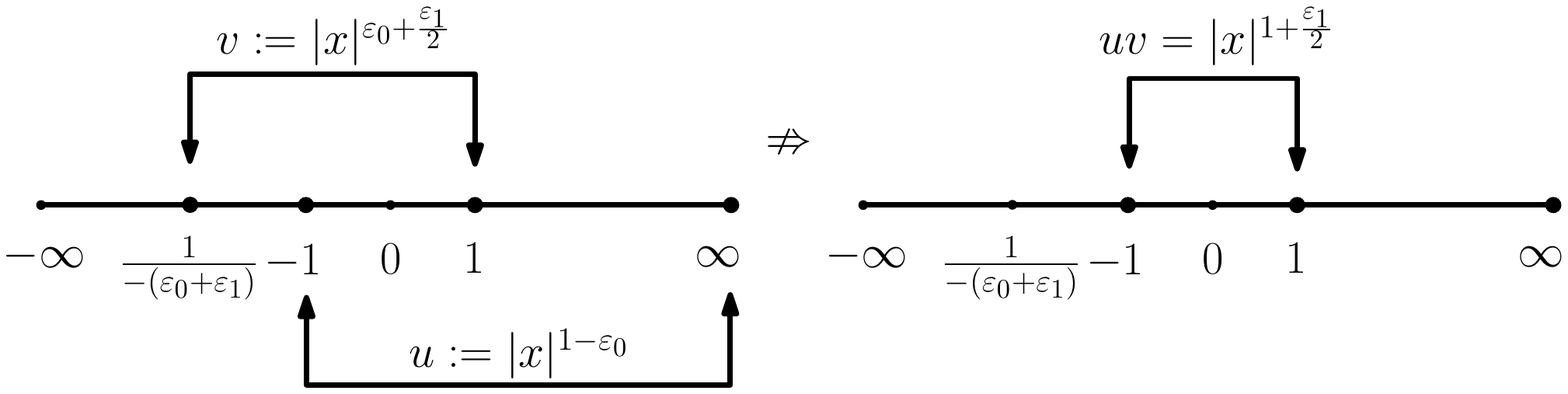}
   \caption{The product will not be in the intersecting reverse class if the reverse class of
   one of the factors fails to touch an infinity.}
      \label{oneinftyproduct}
\end{figure}
\end{example}

\section{Using the visual formalism to illustrate proofs of Harnack's inequality}\label{secmoser}

Fix an open set $\Omega \subset X$. A weight $w$ is said to satisfy Harnack's inequality, with constant $C_H \geq 1$, in $\Omega$ if $w \in RC(-\infty, \infty, C_H)$, that is,
$$
\esssup\limits_{B} w \leq C_H \essinf\limits_{B} w, \quad \forall B \in \B.
$$
That is, $w$ satisfies the most extreme reversal inequality. Visually,

\begin{figure}[H]
   \centering
   \includegraphics[width=3in]{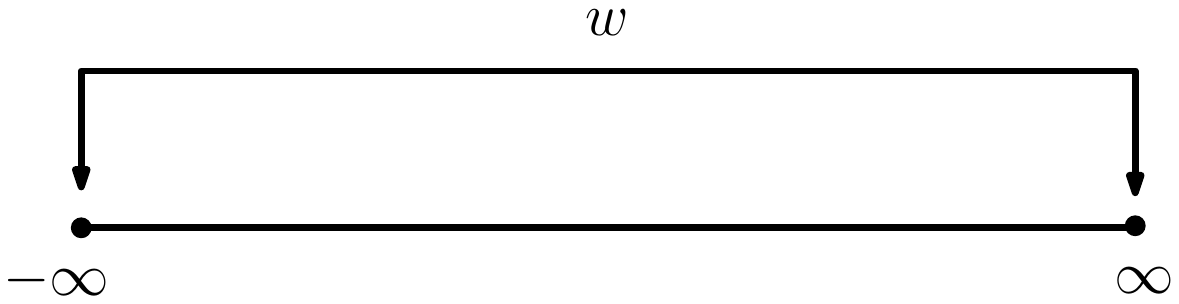}
   \caption{The Harnack class $RC(-\infty, \infty)$.}
      \label{fig:harnackclass}
\end{figure}

In this section we use the visual formalism to provide a brief description of Moser's and Krylov-Safonov's proofs of Harnack's inequality for positive solutions to elliptic PDEs. While Moser's method is most notable for his ingenious
iteration scheme \cite{moserE} in the context of divergence-form operators,
the Krylov-Safonov's method, based on innovative probabilistic tools
\cite{KS79, KS80}, was developed in the context of
non-divergence-form operators. Both of these models stand as
cornerstones in the study of regularity properties of solutions to
PDEs and are flexible enough to be carried out in more general types
of doubling quasi-metric spaces possessing suitable additional
structure (e.g., carrying Sobolev or Poincar\'e-type inequalities). In what follows the underlying space of homogeneous type is Euclidean space $\rn$ with Lebesgue measure (the latter indicated by $| \cdot |$).

\subsection{Moser's iterations and Harnack's inequality}
Let $\Omega \subset \rn$ be an open bounded set and for each $x \in \Omega$ let $A(x)$ be an $n \times n$ symmetric matrix verifying the uniform ellipticity condition
\begin{equation}\label{Aellip}
\Lambda_1 |\xi|^2 \leq \langle A(x)\xi, \xi \rangle \leq \Lambda_2 |\xi|^2, \quad \forall x \in \Omega, \xi \in \rn,
\end{equation}
for some constants $0 < \Lambda_1 \leq \Lambda_2$. Harnack's inequality for positive solutions to the divergence-form
elliptic equation
$$\mathcal{L} u:=\ds\sum_{i,j=1}^n(a_{ij}(x)u_i)_j=\div(A(x)\nabla
u)=0 \text{ in } \Omega\subset\Rn$$ 
was established by J. Moser in
\cite{moserE}. Moreover, Moser showed that the Harnack constant $C_H$ depends only on dimension $n$ and the ratio $\Lambda_2/\Lambda_1$.

The first step in his approach is based on an interaction between a Sobolev inequality and an energy estimate (Caccioppoli's inequality) to show that any positive subsolution $u$ (i.e., $\mathcal{L}u \geq 0$ in $\Omega$) satisfies $u \in RC^{weak}(2, 2\rho)$, where $\rho :=2n/(n-2) > 1$. More precisely,
\begin{equation}\label{MoStep1}
\left(\frac{1}{|B|}\int_B u^{2 \rho}\, dx \right)^{\frac{1}{2\rho}} \leq C(n,\Lambda_2/\Lambda_1) \left( \frac{1}{|B|} \int_{2B} u^{2}\, dx\right)^\frac{1}{2}, \quad \forall B \in \B.
\end{equation}
Then, by means of a finely tuned iterative procedure, illustrated in Figure \ref{moserfig:step1},
\begin{figure}[H] 
   \centering
   \includegraphics[width=3.5in]{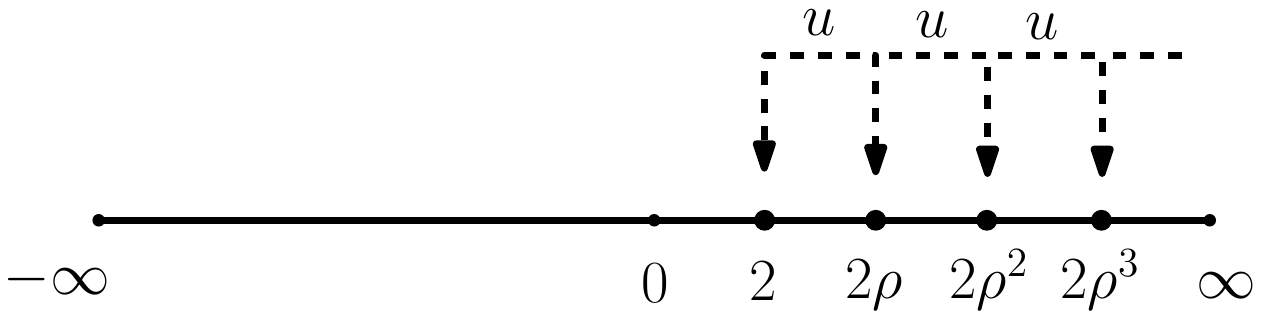}
   \caption{Moser iterations for a positive subsolution $u$.}
   \label{moserfig:step1}
\end{figure}
Moser improved \eqref{MoStep1} to obtain $u \in RC^{weak}(2,\infty)$ as in Figure \ref{moserfig:step2}.
\begin{figure}[H] 
   \centering
   \includegraphics[width=3.5in]{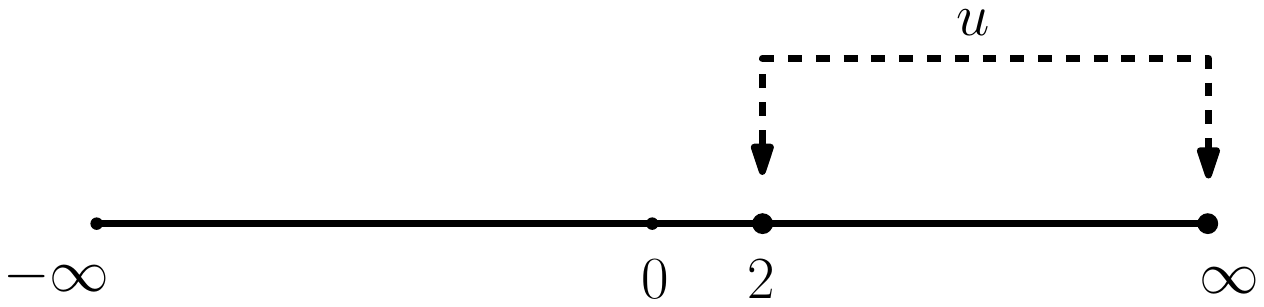}
   \caption{Concatenation and Moser's iterations imply that every positive subsolution $u$ belongs to $RC^{weak}(2,\infty)$.}
   \label{moserfig:step2}
\end{figure}
Consequently, by the self-improving properties for the reverse weak classes (see Figure \ref{fig:selfimproveWRH}), we have $u \in RC^{weak}(p,\infty)$, for every $p > 0$, that is,
\begin{equation}\label{moserite}
\esssup_B u \leq C(p, n, \Lambda_2/\Lambda_1) \left(\frac{1}{|B|}\int_{2B} u^p \, dx\right)^{\frac{1}{p}}, \quad \forall B \in \B.
\end{equation}
as illustrated in Figure \ref{moserfig:possubrc}.
\begin{figure}[H] 
   \centering
   \includegraphics[width=3.5in]{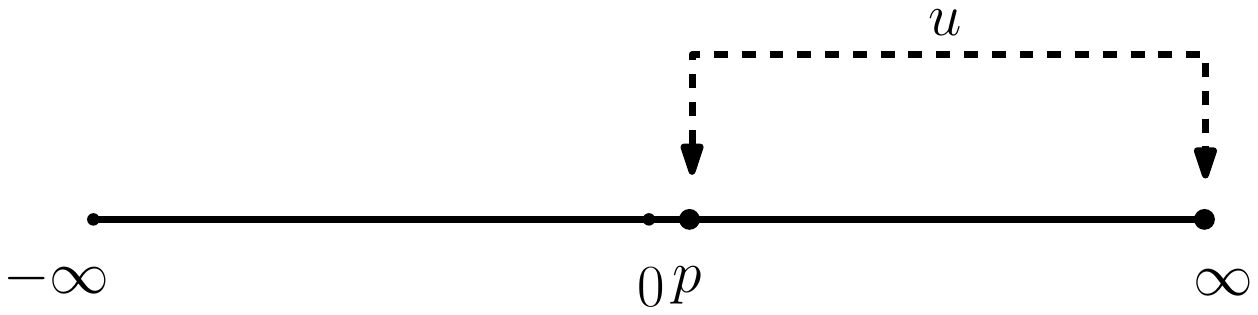}
   \caption{Every positive subsolution $u$ belongs to $RC^{weak}(p,\infty)$ for all $p > 0$.}
   \label{moserfig:possubrc}
\end{figure}
Inequality \eqref{moserite} is usually referred to as the \emph{local boundedness property} for $u$ and it can also be proved by means of De Giorgi's truncations. For both methods, Moser's iterations and De Giorgi's truncations, see for instance  \cite[Section 4.2]{HL}.

Now, if $u$ is a positive solution, then both $u$ and $u^{-1}$ are positive subsolutions and \eqref{moserite} applied to them yields  $u \in RC^{weak}(p,\infty)$ and $u^{-1} \in RC^{weak}(p,\infty)$.  But, by the scaling property with $\theta =-1$, the latter means $u \in RC^{weak}(-\infty,-p)$, that is, for all $p > 0$, 
\begin{equation}\label{moserite2}
\left(\frac{1}{|B|}\int_{2B} u^{-p} \, dx\right)^{-\frac{1}{p}} \leq C(p, n, \Lambda_2/\Lambda_1) \essinf_B u  \quad \forall B \in \B.
\end{equation}
Hence, positive solutions satisfy both \eqref{moserite} and \eqref{moserite2}, this is illustrated in Figure \ref{moserfig:step4}.
\begin{figure}[H] 
   \centering
   \includegraphics[width=3.5in]{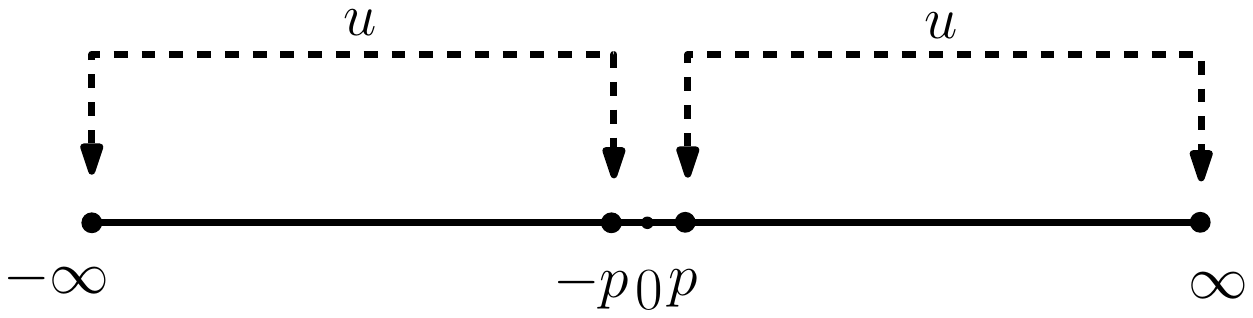}
   \caption{Every positive solution $u$ belongs to $RC^{weak}(-\infty,-p) \cap RC^{weak}(p, \infty)$ for all $p > 0$.  }
   \label{moserfig:step4}
\end{figure}
The next step in Moser's Harnack inequality consists in establishing the existence of $p_0 > 0$ such that if $u$ is a positive supersolution (i.e., $\mathcal{L}u \leq 0$ in $\Omega$), then
$u^{p_0} \in A_2(\Omega)$. This is illustrated  in
\figref{moserfig:step5}.
\begin{figure}[H] 
   \centering
   \includegraphics[width=3.5in]{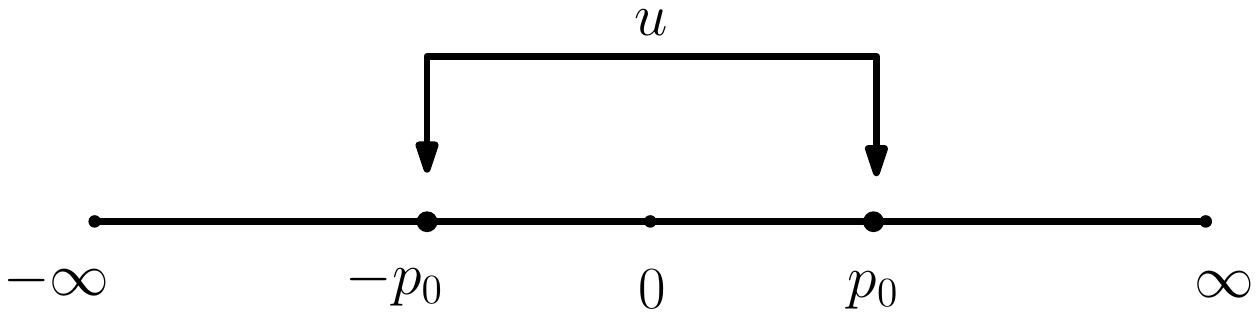}
   \caption{There exists a structural $p_0 > 0$ such that every positive supersolution $u$ (i.e., $\mathcal{L}u \leq 0$) satisfies $u^{p_0} \in A_2$. By the scaling property, this is equivalent to $u \in RC(-p_0, p_0)$.}
   \label{moserfig:step5}
\end{figure}
The step in \figref{moserfig:step5} is typically accomplished by using a Poincar\'e inequality and an energy estimate for $\log u$ to obtain that $\log u \in BMO(\Omega)$. Then, by Corollary \ref{ainfa2} and the techniques in the proof of Corollary \ref{characbmoblo} (\ref{charbmo}) (see also Exercise \ref{RCcrosszero} in Section \ref{secc:ex}), it follows that there is a $p_0 >0$ such that $u^{p_0} \in A_2$.

Finally, by choosing the $p$ in Figure \ref{moserfig:step4} equal to the $p_0$ in  Figure \ref{moserfig:step5} and using the fact that $u^{p_0}$ is doubling, since it is $A_2$, so that dashed lines in Figure \ref{moserfig:step4} turn into solid ones (see Remark \ref{rem:dashedtosolid}), the concatenation property yields $u \in RC(-\infty, -p_0) \cap RC(-p_0, p_0) \cap RC(p_0, \infty) = RC(-\infty, \infty)$.

\subsection{Krylov-Safanov's approach to Harnack's inequality.} As before, let $A(x)$ be a uniformly elliptic matrix satisfying \eqref{Aellip}. During the 1980's, N. Krylov and M. Safonov took recourse to completely new
measure-theoretic tools \cite{KS79, KS80} in order to establish
Harnack's inequality for positive solutions to the non-divergence-form elliptic equation
\begin{equation}\label{nondivform}
L u:=\ds\sum_{i,j=1}^na_{ij}(x)u_{ij}=\tr(A(x)D^2 u)=0  \text{ in } \Omega\subset\Rn.
\end{equation}
Later on, L. Caffarelli's seminal work \cite{C89} on fully non-linear elliptic equations  (see also \cite[Section 4.2]{CaCab}) greatly enriched
and simplified Krylov-Safonov's theory rendering it quite flexible and still
manageable. This, in fact, paved the way for the axiomatizations of 
Krylov-Safonov's theory in doubling quasi-metric spaces carried out in \cite{aft2,
DGL, harnack.maldonado}. All these axiomatic approaches involve, implicitly or explicitly, the
so-called critical density and power-like decay properties.

\begin{dfn} Let $\K_\Omega$ denote a family of non-negative
measurable functions with domain contained in $\Omega$. If $u\in
\K_\Omega$ and $A\subset \text{dom}(u)$ then we write $u\in
\K_\Omega (A),$ where dom$(u)$ stands for the domain of the function
$u$. Assume that $\K_\Omega$ is closed under multiplication by positive scalars. 

Let $0 < \eps < 1 \leq M$. $\K_\Omega$ is said to satisfy the \emph{critical density property} with constants $\eps$ and $M$  if for every $B_{2R}(x_0) \in \B$ and $u \in \K_\Omega(B_{2 R}(x_0))$ with 
\bee
\essinf\limits_{B_{R}(x_0)} u \leq 1, 
\eee
it follows that 
\bee 
\mu(\{x \in B_{2R}(x_0): u(x) > M\}) \leq \eps \mu(B_{2R}(x_0)).
\eee
Let $0 < \varrho < 1 < N$ . $\K_\Omega$ is said to
satisfy the \emph{power-like decay property} with constants $N$ and
$\varrho$ if for every $B_{2R}(x_0) \in \B$ and every $u\in \K_\Omega(B_{2 R}(x_0))$ with \bee
\essinf\limits_{B_{R}(x_0)} u \leq 1, 
\eee 
it follows that 
\bee 
\mu(\{x \in B_{R/2}(x_0): u(x) > N^k\}) \leq \varrho^k \mu(B_{R/2}(x_0))
\quad \forall k \in \na. 
\eee
\end{dfn}
In the context of the non-divergence form elliptic operators \eqref{nondivform}, the critical density property plays the role analogous to the Moser iterations in the divergence-form setting. Indeed, let  $\K_\Omega$ denote the class of supersolutions of \eqref{nondivform} (that is, $u \in \K_\Omega$ if and only if $Lu \leq 0$ in $\Omega$), as a first step it is proved that $\K_\Omega$ possesses the critical density property; see, for instance, Theorem 2.1.1 on page 31 of \cite{guti}. Then, it is proved that if for any given ball $B \in \B$ and any scalar $\lambda$ with $\lambda - u > 0$ in $B$ we have $\lambda - u \in \K_\Omega$, then $u \in RC^{weak}(p,\infty)$ for every $p > 0$. See, for instance \cite[Section 3]{harnack.maldonado}. Therefore, putting these two results together, if $u$ is a positive subsolution ($Lu \geq 0)$, then $L(\lambda - u) = - Lu \leq 0$, so that $\lambda - u \in \K_\Omega$ and, consequently, $u \in RC^{weak}(p,\infty)$ for every $p > 0$. This is illustrated in Figure \ref{moserfig:step3}.
\begin{figure}[H] 
   \centering
   \includegraphics[width=3.5in]{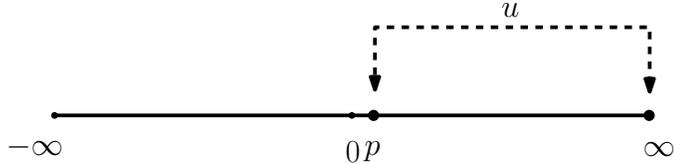}
   \caption{The critical density property implies that every positive  subsolution $u$ (i.e. $Lu \geq 0$) belongs to $RC^{weak}(p,\infty)$ for all $p > 0$.}
   \label{moserfig:step3}
\end{figure}
The next key step in Krylov-Safonov's approach consists in showing that  the class of positive supersolutions possesses the power-like decay property, see for instance Theorem 2.1.3 on page 36 of \cite{guti} and Lemma 4.6 on page 33 of \cite{CaCab}. Now, the fact that the class of supersolutions has the power-like decay  property (and since it is closed under multiplication by positive constants)  amounts to the existence of constants $0 < \delta < 1 \leq C$, depending only on  $\varrho$, $N$, $\Lambda_2/\Lambda_1$, and dimension $n$, such that every supersolution $u$ satisfies the reverse inequality $RC(-\infty,\delta, C)$ (see \cite[Remark~2]{harnack.maldonado}); namely,
\begin{equation}\label{weakH}
\frac{1}{|B|}\int_{B}u^\dl\,dx \leq C^\delta \essinf_{B} u^\dl \quad \forall B \in \B.
\end{equation}
In other words, $u^\delta \in A_1(\Omega)$. Inequality \eqref{weakH} is usually referred to as the \emph{weak Harnack inequality} for $u$ and it is here illustrated in \figref{fig:krylovsafonov}.
\begin{figure}[H]
\centering
\includegraphics[width=3.5in]{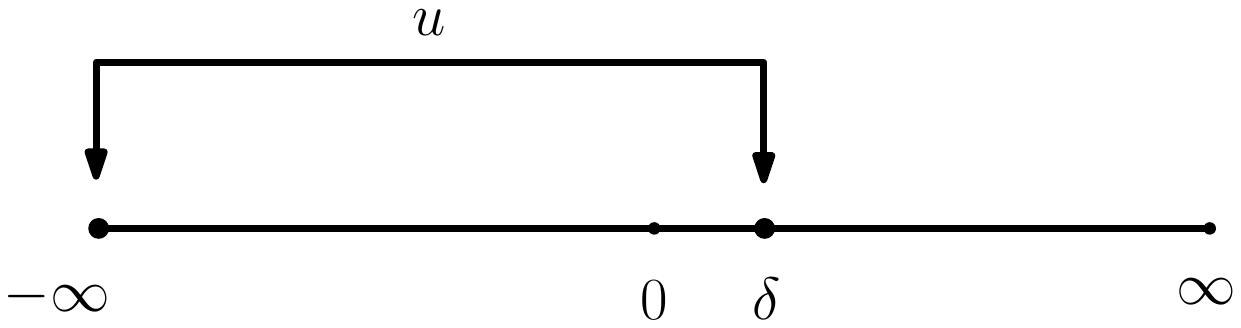}
\caption{There exists a structural $\delta > 0$ such that every positive supersolution $u$ (i.e., $Lu \leq 0$) satisfies $u^\delta \in A_1$. By the scaling property, this is equivalent to $u \in RC(-\infty, \delta)$.}
 \label{fig:krylovsafonov}
\end{figure}

Finally, if $u$ is a positive solution in $\Omega$, then it is both a subsolution and a supersolution and by choosing the $p$ in Figure \ref{moserfig:step3} equal to the $\delta$ in Figure \ref{fig:krylovsafonov} (and using the fact that $u^\delta$ is doubling, since it is $A_1$, turns the dashed arrow in Figure \ref{moserfig:step3} into a solid one), the concatenation property yields $u \in RC(-\infty,\delta) \cap RC(\delta, \infty) = RC(-\infty, \infty)$.

\section{Further practice}\label{secc:ex}

\begin{exe} Prove that the three axioms of the visual formalism for the reverse classes $RC(r,s)$ also apply to the weak reverse classes $RC^{weak}(r,s)$ and study the behavior of the weak-reversal constants.

\end{exe}

\begin{exe} Use the visual formalism to prove that if $w \in A_\infty$ and $w^r \in A_1$ for some $0 < r < \infty$, then $w \in A_1$. Estimate $[w]_{A_1}$ in terms of $[w]_{A_\infty}$ and $[w^r]_{A_1}$ by considering the cases $r > 1$ and $0 < r < 1$.

\end{exe}

\begin{exe} Fix $1 < s < \infty$. Use the visual formalism to prove that  $w^s \in A_\infty$ implies $w \in RH_s$ with $[w]_{RH_s} \leq [w^s]_{A_\infty}^{1/s}$.

\end{exe}

\begin{exe} Prove Theorem \ref{factorization} using the visual formalism.

\end{exe}

\begin{exe}\label{RCcrosszero} Use the techniques depicted in Figures \ref{fig:bmorspositive} and \ref{fig:blornegative} to prove that every weight satisfying a reverse inequality ``must cross the exponent zero''. That is, if $w \in RC(r,s)$ for some $-\infty \leq r < s < 0$, then $w \in RC(r,\eps)$ for some $\eps> 0$.  Similarly, if $w \in RC(r, s)$ for some $0< r < s \leq \infty$, then $w \in RC(-\eps, s)$ for some $\eps > 0$. Consequently, every reverse class $RC(r,s)$ self-improves ``to touch zero''. Meaning that whenever  $-\infty \leq r < s < 0$, then $RC(r,s) \subset RC(r,0)$. Similarly, if $0< r < s \leq \infty$, then  $RC(r,s) \subset RC(0,s)$.
\end{exe}

\begin{exe}\label{exe:charabuo}  Adapt the arguments in Figures \ref{fig:blorpositive} and \ref{fig:blornegative} to visually prove Corollary \ref{characbmoblo} \eqref{charbuo}.
\end{exe}

\begin{exe} Use Figures \ref{blo} and  \ref{product}  to visually prove Corollary \ref{bmoblo}. In the same vain, use Figures \ref{blo} and \ref{product} to visually prove that $BMO = BUO - BUO$.

\end{exe}

\section{Acknowledgements}

The authors would like to thank David Cruz-Uribe as well as the anonymous referees for their thorough reading of the manuscript and the numerous suggestions that improved its presentation.

\end{document}